\documentclass[letterpaper]{article}

\usepackage[margin = 1in]{geometry}
\usepackage{cite}

\usepackage[usenames, dvipsnames]{color}
\usepackage{hyperref}
\hypersetup
{
    colorlinks,	%
    citecolor=violet,%
    filecolor=violet,%
    linkcolor=violet,%
    urlcolor=violet
}

\usepackage{bbm}
\usepackage{amsmath}
\usepackage{amssymb}
\usepackage{amsthm}

\usepackage{tikz}
\usepackage{tikz-cd}

\usepackage{mathtools}
\usetikzlibrary{matrix,arrows}
\tikzset{commutative diagrams/diagrams={ampersand replacement=\&}}

\tikzset{dbl/.style={double,
		double equal sign distance,
		-implies,
		shorten >=10pt,
		shorten <=10pt}}

\tikzset{
	invisible/.style={opacity=0},
	visible on/.style={alt={#1{}{invisible}}},
	alt/.code args={<#1>#2#3}{%
		\alt<#1>{\pgfkeysalso{#2}}{\pgfkeysalso{#3}}%
	}
}

\numberwithin{equation}{section}

\newtheorem{theorem}[equation]{Theorem}
\newtheorem{corollary}[equation]{Corollary}
\newtheorem{lemma}[equation]{Lemma}

\newtheorem{definition}[equation]{Definition}

\newcommand{\R}{\mathbb{R}}

\newcommand{\X}{\mathbb{X}}
\newcommand{\Y}{\mathbb{Y}}

\newcommand{\M}{\mathbb{M}}

\newcommand{\CC}{\mathcal{C}}
\newcommand{\DD}{\mathcal{D}}

\newcommand{\HH}{\mathcal{H}}
\newcommand{\II}{\mathcal{I}}

\newcommand{\PP}{\mathcal{P}}

\newcommand{\LL}{\mathcal{L}}
\newcommand{\EE}{\mathcal{E}}
\renewcommand{\SS}{\mathcal{S}}
\newcommand{\KK}{\mathcal{K}}
\newcommand{\TT}{\mathcal{T}}
\newcommand{\UU}{\mathcal{U}}
\newcommand{\QQ}{\mathcal{Q}}

\newcommand{\VV}{\mathcal{V}}

\newcommand{\Ffunc}{\mathbf{F}}
\newcommand{\Gfunc}{\mathbf{G}}

\newcommand{\Hfunc}{\mathbf{H}}
\newcommand{\Tfunc}{\mathbf{T}}

\newcommand{\Int}{\mathbf{Int}}
\newcommand{\Open}{\mathbf{Open}}
\newcommand{\Set}{\mathbf{Set}}

\newcommand{\Vect}{\mathbf{Vect}}
\newcommand{\Cat}{\mathbf{Cat}}
\newcommand{\CAT}{\mathbf{CAT}}
\newcommand{\ACT}{\mathbf{ACT}}

\newcommand{\Top}{\mathbf{Top}}

\newcommand{\Trans}{\mathbf{Trans}}

\newcommand{\End}{\mathbf{End}}

\newcommand{\Hom }{\mathrm{Hom}}

\newcommand{\Obj}{\mathrm{obj }\; }

\newcommand{\inv}{^{-1}}
\newcommand{\e}{\varepsilon}
\renewcommand{\phi}{\varphi}

 \title{Theory of interleavings on categories with a flow}
 \author{Vin de Silva\footnote{Department of Mathematics, Pomona College},
 Elizabeth Munch\footnote{Department of Computational Mathematics, Science and Engineering (CMSE),	Department of Mathematics, Michigan State University},
 and Anastasios Stefanou\footnote{Department of Mathematics and Statistics, University at Albany, SUNY}}
 \date{}

\begin{document}
\maketitle

\begin{abstract}
The interleaving distance was originally defined in the field of Topological Data Analysis (TDA) by Chazal et al.\ as a metric on the class of persistence modules parametrized over the real line.
Bubenik et al.\ subsequently extended the definition to categories of functors on a poset, the objects in these categories being regarded as `generalized persistence modules'.
These metrics typically depend on the choice of a lax semigroup of endomorphisms of the poset.
The purpose of the present paper is to develop a more general framework for the notion of interleaving distance using the theory of `actegories'.
Specifically, we extend the notion of interleaving distance to arbitrary categories
equipped with a flow, i.e.~a lax monoidal action by the monoid $[0,\infty)$.
In this way, the class of objects in such a category acquires the structure of a Lawvere metric space.
Functors that are colax $[0,\infty)$-equivariant yield maps that are $1$-Lipschitz.
This leads to concise proofs of various known stability results from TDA, by considering appropriate colax $[0,\infty)$-equivariant functors.
Along the way, we show that several common metrics, including the Hausdorff distance and the $L^{\infty}$-norm, can be realized as interleaving distances in this general perspective.

\end{abstract}

\bigskip
\bigskip

\tableofcontents

%---------------------------------------------------------------
\section{Introduction}
Behind every data analysis tool is an implicit reliance on metrics between the data points.
If the data points are embedded in some metric space, such as Euclidean space, we use the distance inherited from the space to describe proximity.
In particular, clustering arises by looking for groups of data points which are ``close'' in some chosen metric, but far in that metric from other data points.
The tools of Topological Data Analysis exploit this idea of studying a collection of points with a metric (that is, finite metric spaces) by constructing topological signatures which represent some aspect of the data, and using these signatures as proxies for the original data sets.
Some commonly used topological signatures include persistence diagrams \cite{Edelsbrunner2002}, persistence modules \cite{Zomorodian2004,Chazal2009b}, Reeb graphs \cite{Reeb1946,deSilva2016}, and Mapper \cite{Singh2007}.
Indeed, arguably the most powerful theorem in TDA is the stability theorem of Cohen-Steiner et al.\ \cite{Cohen-Steiner2007,Chazal2009b} which states that for a certain choice of metric\footnote{The bottleneck distance} on the persistence diagrams arising from point clouds, the distance between the signatures can be no greater than the Hausdorff distance between the point clouds from which they were constructed.
In other words, the persistence diagram is statistically robust with respect to perturbations that are small in the Hausdorff distance.

The stability theorem falls naturally into two parts~\cite{Chazal2009b,Bubenik2015}: both of the transformations in the sequence
\[
\text{data}
\stackrel{\rm (i)}{\longrightarrow}
\text{persistence module}
\stackrel{\rm (ii)}{\longrightarrow}
\text{persistence diagram}
\]
are individually 1-Lipschitz.\footnote
{Strictly speaking, there are many possible choices of transformation~(i), and most of the standard ones are Lipschitz for some specific constant, usually 1 or 2.}
This factorization was not noticed for a while, partly because the original algorithm for constructing persistence diagrams from data~\cite{Edelsbrunner2002} proceeded directly without reference to persistence modules.
Chazal et al.~\cite{Chazal2009b} were the first to draw attention to this division of labor, defining the `{interleaving distance}' between persistence modules which makes possible to contemplate Lipschitzity for the two maps separately.
In \cite{Bubenik2015}, the stability of parts (i) and (ii) are respectively called `soft' stability and `hard' stability: part (i) operates at an abstract algebraic/categorical level, while part (ii) requires a detailed study of the relationship between persistence modules and their diagrams \cite{Cohen-Steiner2007,Chazal2009b,Chazal2016,Bauer2014}.

In this paper we study generalizations of the interleaving distance, and of part (i) of the stability theorem.
Let us recall the main concepts.
A persistence module is a 1-parameter diagram of vector spaces and linear maps; most concisely it is a functor
\[
\Ffunc: (\R,\leq) \to \Vect
\]
from the real line (viewed as a poset category) to the category of vector spaces over some field.
These are typically obtained in TDA by constructing, from data, a 1-parameter nested family of simplicial complexes (perhaps approximating the finite data set at different scales) and applying a homology functor with field coefficients.
The persistence diagram is a representation of the structure of the rank function
\[
r^s_t = \operatorname{rank} [\Ffunc(s) \to \Ffunc(t)]
\]
by a collection of pairs $(b,d)$ where $b \leq d$, so that $r^s_t$ (roughly speaking) counts those pairs for which $b \leq s \leq t \leq d$.
We will say nothing more about persistence diagrams.

Now let us compare two persistence modules $\Ffunc,\Gfunc: (\R,\leq) \to \Vect$.
We consider them to be `the same' if there exist natural transformations $\phi:\Ffunc \Rightarrow \Gfunc$ and $\psi:\Gfunc \Rightarrow \Ffunc$ such that $\phi\psi = I_{\Gfunc}$ and $\psi\phi = I_{\Ffunc}$;
that is, if there is an isomorphism between them.
Chazal et al.~\cite{Chazal2009b} extend this idea to the notion of an $\e$-interleaving, thought of as an $\e$-approximate isomorphism.
This is a pair of natural transformations $\phi:\Ffunc \Rightarrow \Gfunc\Tfunc_\e$ and $\psi:\Gfunc \Rightarrow \Ffunc\Tfunc_\e$,
where $\Tfunc_\e : (\R,\leq) \to (\R,\leq)$ is a functor together with a natural transformation $\eta_\e:I_{(\R,\leq)}\Rightarrow\Tfunc_{\e}$, called \textbf{a translation functor}, defined by $a \mapsto a + \e$,
such that the diagrams
\begin{equation*}
\begin{tikzcd}
\Ffunc
\arrow[d, Rightarrow, "{\Ffunc\eta_{\e}}"']
\&
\Gfunc
\arrow[dl, Rightarrow, "{\psi}"' very near start]
\arrow[d, Rightarrow, swap, "{\Gfunc\eta_{\e}}"']
\\
\Ffunc\Tfunc_\e
\arrow[d, Rightarrow,  "{\Ffunc\eta_{\e}\Tfunc_{\e}}"']
\&
\Gfunc\Tfunc_\e
\arrow[Leftarrow,  ul, crossing over, "{\varphi}"' very near end]
\arrow[d, Rightarrow, swap, "{\Gfunc\eta_{\e}\Tfunc_{\e}}"']
\arrow[dl, Rightarrow,  "{\psi\Tfunc_\e}"  near start]
\\
\Ffunc\Tfunc_{2\e}\&
\Gfunc\Tfunc_{2\e}
\arrow[Leftarrow,  ul, crossing over, "{\phi\Tfunc_\e}"  near end]
\end{tikzcd}
\end{equation*}
commute.
The interleaving distance between $\Ffunc,\Gfunc$ is simply defined to be the infimum of those values $\e$ for which an $\e$-interleaving exists.

This beautiful and powerful idea was extended by Bubenik et al.~\cite{Bubenik2014,Bubenik2015} (see also Lesnick~\cite{Lesnick2015}), to general functor categories $\DD^\PP$, for $\PP$ a poset category and $\DD$ an arbitrary category.
In that work, interleavings are defined with respect to a particular collection of endomorphisms on $\PP$, again called `translations'.
The easiest approach is to select a preferred 1-parameter family of translations $(\Tfunc_\e)$ and define $\e$-interleavings with respect to those.
This general approach turns out to be quite fruitful.
If $\PP = (\R,\leq)$ and $\DD = \Vect$ then we recover the original interleaving distance on persistence modules.
If $\PP = (\R,\leq)$ and $\DD = \Set$, then the objects of $\DD^\PP$ can be thought of as `merge trees' and we recover a metric defined by Morozov et al.~\cite{morozov2013interleaving,AnastasiosThesis}.
And if $\PP = \Int$, the poset category of real open intervals with $\Tfunc_\e$ being the operation that thickens an interval by~$\e$ on each side and $\DD = \Set$, then the objects of $\DD^\PP$ can be interpreted as Reeb graphs and we recover the metric defined in~\cite{deSilva2016}.%
\footnote{Strictly speaking, merge trees and Reeb graphs correspond to objects in subcategories of their respective functor categories, specified by suitable regularity conditions (and a cosheaf hypothesis in the case of Reeb graphs).}

In this paper, we extend the idea of interleavings yet further to be defined on arbitrary categories $\CC$ with the additional structure of a coherent $[0,\infty)$-action.
These categories are sometimes called $[0,\infty)$-actegories.\footnote
{There are many different types of actegory \cite{janelidze2001note}, so the meaning of the word is not well-specified.}
However, in this paper any coherent $[0,\infty)$-action will be called a \textbf{flow} for simplicity.\footnote
{A preprint of this paper referenced the `categories with a flow' construction as `$[0,\infty)$-actegories' \cite{deSilva2017}.  }
% The authors decided to change the name `$[0,\infty)$-actegories' \cite{deSilva2017} to `categories with a flow' so as to ease the way with which we can talk about these objects, and better imply the behaviors of the objects.}
We give precise definitions in due course.
The upshot of this work is that categories with a flow inherit the structure of a \textbf{symmetric Lawvere metric space}: there is a map $d_{\CC}:\Obj\CC\times \Obj\CC\to [0,\infty]$, satisfying the relations $d_{\CC}(a,a)=0$, and $d_{\CC}(a,b) = d_{\CC}(b,a)$, and $d_{\CC}(a,c)\leq d_{\CC}(a,b)+d_{\CC}(b,c)$.%
\footnote{Lawvere metric spaces (in the absence of the symmetry relation) may be thought of as categories enriched over the symmetric monoidal category $([0,\infty], \geq, +)$, through the equation $\hom_\CC(a,b) = d_\CC(a,b)$.}

One pleasant outcome is that many commonly used metrics can be realized as interleaving distances.
We will show that these include the Hausdorff distance, the $L^\infty$ distance on $\R^n$, and the extended $L^\infty$ distance on $\R$-spaces and $\M$-spaces in general, where $\M$ is any metric space.
A final bonus is that we can retrieve several of the usual soft stability theorems as special instances of a single theorem, which asserts that functors between categories with a flow that are colax $[0,\infty)$-equivariant give rise to maps which are 1-Lipschitz.
While those original theorems are not difficult to prove, we find it illuminating to view those theorems through the unified viewpoint developed here.

\paragraph{Outline}
In Section~\ref{sec:CategoriesWithLaxAction}, we define categories with a flow and define the interleaving distance.
In Section~\ref{Sec:Examples}, we show that several common metrics are interleaving distances.
In Section~\ref{Sec:Stability}, we define colax $[0,\infty)$-equivariant functors between categories with a flow and show that they give rise to 1-Lipschitz maps.
Various soft stability results from TDA are deduced from this.
In Section~\ref{sec:MetaTheorem}, we show that the interleaving process is functorial, and we explain how to view categories with a flow and colax $[0,\infty)$-equivariant functors in terms of higher category theory.

%\subsection*{Acknowledgments}
%The authors would like to express their thanks to Peter Bubenik, Tom Leinster, Justin Curry, Amit Patel, and Marco Varisco for many helpful discussions and suggestions during the course of this work.

%---------------------------------------------
\section{Interleavings on categories with a flow}
\label{sec:CategoriesWithLaxAction}
In this section we define categories with a flow and show that these categories are symmetric Lawvere spaces.
Specifically, we show that every flow $\TT$ on a category $\CC$ induces an extended pseudometric $d_{(\CC,\TT)}$ on $\CC$ called the interleaving distance.
Our construction extends the definition of the interleaving distance from the context of functor categories $\DD^{\PP}$ of generalized persistence modules \cite{Bubenik2015}, to the context of arbitrary categories with a flow $\CC$.

\subsection{A review on actegories}
A monoidal category $\VV=(\VV,\otimes,I,a,\ell,r)$ is a category with a notion of a tensor product.
A lax monoidal functor $\Ffunc:\VV\to\VV'$ between monoidal categories $\VV=(\VV,\otimes,I,a,\ell,r)$  and $\VV'=(\VV',\otimes',I',a',\ell',r')$ consists of a triple $\Ffunc=(\Ffunc,u,\mu)$ where $\Ffunc :\VV\to\VV'$
is an ordinary functor, $\mu$ is a natural transformation with components $\mu_{x,y}$ : $\Ffunc(x) \otimes' \Ffunc(y)\Rightarrow \Ffunc(x \otimes y )$, and $x,y\in\Obj\VV$ and $u: I'\Rightarrow \Ffunc (I)$ is a natural transformation.
These data are such that the diagrams
\[
\begin{tikzcd}
((\Ffunc(x)\otimes' \Ffunc(y))\otimes' \Ffunc(z) \arrow[dd, swap, "{a'_{\Ffunc(x),\Ffunc(y),\Ffunc(z)}}"]\arrow[r,  "{\mu_{x,y}\otimes' 1_{\Ffunc(z)}}"] \&\Ffunc(x\otimes y)\otimes' \Ffunc(z)\arrow[r,   "{\mu_{x\otimes y,z}}"]\& \Ffunc((x\otimes y)\otimes z)\arrow[dd,  "{\Ffunc(a_{x,y,z})}"] \\
\\
\Ffunc(x)\otimes' (\Ffunc(y)\otimes' \Ffunc(z)) \arrow[r, swap, "{1_{\Ffunc(x)}\otimes'\mu_{y\otimes z}}"]\& \Ffunc(x)\otimes'\Ffunc(y\otimes z)\arrow[r, swap, "{\mu_{x,y\otimes z}}"]\& \Ffunc(x\otimes (y\otimes z))
\end{tikzcd}
\]
and
\[
\begin{tikzcd}
I'\otimes' \Ffunc(x) \arrow[dd, swap, "{\ell'_{\Ffunc(x)}}"]\arrow[rr,  "{u\otimes' 1_{\Ffunc(x)}}"] \&\&\Ffunc(I)\otimes'\Ffunc(x)\arrow[dd,   "{\mu_{I, x}}"] \\
\\
\Ffunc(x)\&\&\Ffunc(I\otimes x) \arrow[ll, swap, "{\Ffunc(\ell_{x})}"]
\end{tikzcd}
\]
and
\[
\begin{tikzcd}
\Ffunc(x)\otimes'I' \arrow[dd, swap, "{r'_{\Ffunc(x)}}"]\arrow[rr,  "{1_{\Ffunc(x)}\otimes' u}"] \&\&\Ffunc(x)\otimes'\Ffunc(I)\arrow[dd,   "{\mu_{x, I}}"] \\
\\
\Ffunc(x)\&\&\Ffunc(x\otimes I) \arrow[ll, swap, "{\Ffunc(r_{x})}"]
\end{tikzcd}
\]
commute for all the objects involved.
Note that for any category $\CC$, the category of endofunctors $\End(\CC)=\CC^{\CC}$ is monoidal.
In this case, the horizontal composition is used for the composition of functors and the Godement product for the composition of natural transformations.

Given a monoidal category $\VV=(\VV,\otimes,\II)$ a \textbf{coherent action of $\VV$ on $\CC$} is a lax monoidal functor $\VV\to\End(\CC)$ of monoidal categories \cite{Kelly1974,janelidze2001note}.
Actions of monoidal categories appeared firstly in a paper of Benabou \cite{benabou1967introduction} and then in Pareigis \cite{pareigis1977non} (Street has suggested the term actegories).
A \textbf{$\VV$-actegory} is a category $\CC$ together with a coherent action of $\VV$ on $\CC$.
A  morphism of $\VV$-actegories, called a \textbf{colax $\VV$-equivariant functor}, is a functor $\HH:\CC\to\DD$ that commutes with the coherent actions of $\CC$ and $\DD$ up to coherent natural transformations.

\subsection{Categories with a flow}
\label{Ssec:Definitions}
%I use slopypar to avoid wrapped equations
Consider the monoidal category $(([0,\infty),\leq),+,0)$ whose tensor product is given by the addition operation:
\begin{align*}
    	+:([0,\infty),\leq)\times ([0,\infty),\leq)&\to([0,\infty),\leq).\\
    	(\e,\zeta)&\mapsto\e+\zeta \\
    	((\e\leq\e'),(\zeta\leq\zeta'))&\mapsto(\e\leq\e')+(\zeta\leq\zeta'):=(\e+\zeta\leq \e'+\zeta').
\end{align*}
and the tensor unit is given by the zero element $0$.

\begin{definition}
\label{definition:catWithLaxAction}
A coherent $[0,\infty)$-action $\TT=(\TT,u,\mu)$ on a category $\CC$ is said to be a \textbf{flow}.
Specifically, a flow $\TT$ on a category $\CC$ consists of
\begin{itemize}
\item a functor $\TT:([0,\infty),\leq)\to\End(\CC)$, $\e\mapsto\TT_\e$,
\item a natural transformation $u:I_{\CC}\Rightarrow\TT_{0}$, where $I_{\CC}$ is the identity endofunctor of $\CC$, and
\item a collection of natural transformations $\mu_{\e,\zeta}:\TT_{\e}\TT_{\zeta}\Rightarrow\TT_{\e+\zeta}$, $\e,\zeta\geq0$,
\end{itemize}
such that the diagrams
\begin{equation*}
\begin{tikzcd}
\& \TT_{\e}\arrow[ld, swap, Rightarrow, "{u I_{\TT_{\e}}}"]\arrow[rd,  equal, ] \& \&
\& \TT_{\e}\arrow[ld, swap, Rightarrow, "{I_{\TT_{\e}} u}"]\arrow[rd,  equal, ]\\
\TT_{0}\TT_{\e}  \arrow[rr, Rightarrow, "{\mu_{0,\e}}"] \& \& \TT_{\e} \&
\TT_{\e}\TT_{0}  \arrow[rr, Rightarrow, "{\mu_{\e,0}}"] \& \& \TT_{\e}
\\
\TT_{\e}\TT_{\zeta}\TT_{\delta}  \arrow[rr, Rightarrow, "{I_{\TT_{\e}}\mu_{\zeta,\delta}}"]\arrow[dd, swap, Rightarrow, "{\mu_{\e,\zeta}I_{\TT_{\delta}}}"]\&\&\TT_{\e}\TT_{\zeta+\delta}\arrow[dd, Rightarrow, "{\mu_{\e,\zeta+\delta}}"] \&
\TT_{\e}\TT_{\zeta}  \arrow[rr, Rightarrow, "{\mu_{\e,\zeta}}"]\arrow[dd, swap, Rightarrow, "{\TT_{(\e\leq \delta)}\TT_{(\zeta\leq \kappa)}}"' near start]\&\& \TT_{\e+\zeta} \arrow[dd, Rightarrow, "{\TT_{(\e+\zeta\leq \delta+\kappa)}}"' near end]
\\
\\
\TT_{\e+\zeta}\TT_{\delta}\arrow[rr, swap, Rightarrow, "{\mu_{\e+\zeta,\delta}}"] \& \& \TT_{\e+\zeta+\delta}  \&
\TT_{\delta}\TT_{\kappa}\arrow[rr, swap, Rightarrow, "{\mu_{\delta,\kappa}}"'] \& \& \TT_{\delta+\kappa}
\end{tikzcd}
\end{equation*}
commute for every $\e,\zeta,\delta,\kappa\geq0$.
A category $\CC$ with a flow $\TT$ is denoted by $(\CC,\TT)$.
\end{definition}

\begin{definition}
Let a category with a flow $(\CC,\TT)$ be given.
Then for each $\e\geq0$, we call the endofunctor $\TT_{\e}:\CC\to\CC$ the \textbf{$\e$-translation} of $\CC$.
We call $u$ and $\mu_{\e,\zeta}$ the \textbf{coherence natural transformations}.
If the coherence natural transformations are all identities, the flow is called \textbf{strict}.
\end{definition}
A strict flow $\TT$ on a category $\CC$ is a functor $\TT:[0,\infty)\to\End(\CC)$, $\e\mapsto\TT_{\e}$, such that $\TT_0=I_{\CC}$ and $\TT_{\e}\TT_{\zeta}=\TT_{\e+\zeta}$, for all $\e,\zeta\geq0$.

\subsection{Interleavings on categories with a flow}
Given a category $\CC$ a flow $\TT$ on $\CC$ enables us to measure `how far' two objects in $\CC$ are from being isomorphic up to a coherence natural transformation.

\begin{definition}\label{definition:interleavings}
Let $a,b$ be two objects in $\mathcal{C}$.
A \textbf{weak $\e$-interleaving of $a$ and $b$}, denoted $(\varphi,\psi)$, consists of a pair  of morphisms $\varphi:a\to\TT_{\e}b$ and $\psi:b\to\TT_{\e}a$ in $\CC$ such that the diagrams
\begin{equation}
\label{eq:interleaving}
\begin{tikzcd}%$[row sep=huge, column sep=huge]
\TT_{0}a
\arrow[dd, "{\TT_{(0\leq2\e),a}}"']
\&
a
\arrow[l, "{u_{a}}"']
\&
b
\arrow[r,"{u_{b}}"]
\arrow[dl,  "{\psi}"' very near start]
\&
\TT_{0}b
\arrow[dd, swap, "{\TT_{(0\leq2\e),b}}"']
\\
\&
\TT_{\e}a
\&
\TT_{\e}b
\arrow[leftarrow,  ul, crossing over, "{\varphi}"' very near end]
\arrow[dl,  "{\TT_{\e}\psi}" very near start]
\\
\TT_{2\e}a
\&
\TT_{\e}\TT_{\e}a
\arrow[l, "{\mu_{\e,\e,a}}"' ]
\&
\TT_{\e}\TT_{\e}b
\arrow[leftarrow,  ul, crossing over, "{\TT_{\e}\phi}" very near end]
\arrow[r, swap, "{\mu_{\e,\e,b}}"' ]
\&
\TT_{2\e}b
\end{tikzcd}
\end{equation}
commute.
We say that $a,b$ are \textbf{weakly $\e$-interleaved} if there exists a weak $\e$-interleaving $(\varphi,\psi)$ of $a$ and $b$.
The \textbf{interleaving distance with respect to $\TT$} for a pair of objects $a,b$ in $\CC$ is defined to be
$$d_{(\CC,\TT)}(a,b)=\inf\{\e\geq0 \mid a,b\text{ are weakly }\e\text{-interleaved}\}.$$
If $a$ and $b$ are not weakly interleaved for any $\e$, we set $d_{(\CC,\TT)}(a,b) = \infty$.
\end{definition}
We use the term ``weakly'' to distinguish Definition~\ref{definition:interleavings} from the interleavings in the restricted setting, Definition~\ref{definition:BubenikInterleaving}, which will be further discussed in Section~\ref{Ssec:InterleavingsGPMs}.

\begin{theorem}
\label{theorem:interleavingExtPseudometric} $d_{(\CC,\TT)}$ is an extended pseudometric on $\Obj\CC$.
\end{theorem}
\begin{proof}
For the sake of brevity in this proof, we write $d = d_{(\CC,\TT)}$.
It is clear by definition that $d$ is symmetric.
Setting $\phi = \psi = u_a$ gives a 0-interleaving of $a$ with itself, hence $d(a,a)=0$ for any object $a$ in $\CC$.

Next, we show that the triangle inequality holds.
Let $a,b,c \in \CC$. If either $d(a,b) = \infty$ or $d(b,c) = \infty$, then trivially $d(a,c) \leq d(a,b) + d(b,c)$.
Now, suppose that for some $0 \leq \e, \zeta < \infty$, the objects $a,b$ are $\e$-interleaved via $(\phi,\psi)$ and the objects $b,c$ are $\zeta$-interleaved via $(\varphi',\psi')$.
Define
$\phi'':a\to\TT_{\e+\zeta}c$
and
$\psi'':c\to\TT_{\e+\zeta}a$
by
\begin{equation*}
\begin{tikzcd}
a
\arrow[r, "{\varphi}"]
\arrow[rrr, bend left, "{\varphi''}"]
\&
\TT_{\e}b
\arrow[r, "{\TT_{\e}\varphi'}"]
\&
\TT_{\e}\TT_{\zeta}c
\arrow[r, "{\mu_{\e,\zeta,c}}"]
\&
\TT_{\e+\zeta}c
\end{tikzcd}
\text{ and }
\begin{tikzcd}
c
\arrow[r, "{\psi'}"]
\arrow[rrr, bend left, "{\psi''}"]
\&
\TT_{\zeta}b
\arrow[r, "{\TT_{\zeta}\psi}"]
\&
\TT_{\zeta}\TT_{\e}a
\arrow[r, "{\mu_{\zeta,\e,a}}"]
\&
\TT_{\zeta+\e}a
\end{tikzcd}
\end{equation*}
respectively.
We claim that $(\varphi'',\psi'')$ is an $(\e+\zeta)$-interleaving of $a$ and $c$.
Showing that the left half of Diagram~\ref{eq:interleaving} commutes means showing
\begin{equation}
\label{eq:BigDgmPerimter}
\mu_{\e+\zeta,\e+\zeta,a}\circ\TT_{\e+\zeta}[\psi'']\circ\varphi''=\TT_{(0\leq2(\e+\zeta)),a}\circ u_{a}.
\end{equation}
Indeed,
via functoriality and the definition of interleaving,
we have the following commutative diagram with Equation~\ref{eq:BigDgmPerimter} as the perimeter.
\begin{equation*}
\begin{tikzpicture}[baseline= (a).base]
\node[scale=.7] (a) at (0,0){
\begin{tikzcd}[row sep=huge, column sep=huge]
\TT_{0}a
\arrow[dd, "{\TT_{(0\leq2\e),a}}"' description]
\ar[dddddd, bend right, "{\TT_{(0\leq2(\e+\zeta)),a}}" description]
\&
\&
a
\arrow[ll, "{u_{a}}"']
\arrow[dr, swap, "{\varphi}"' description]
\ar[dddrrr, bend left, "{\phi''}"]
\\
\&
\&
\&
\TT_{\e}b
\arrow[ld, "{\TT_{\e}\psi}"']
\arrow[d, "{\TT_{\e}u_{b}}"']
\arrow[rd, swap, "{\TT_{\e}\varphi'}"' description]
\\
\TT_{2\e}a
\arrow[dddd, "{\TT_{(2\e\leq2(\e+\zeta)),a}}"' description]
\&
\TT_{\e}\TT_{\e}a
\arrow[l, "{\mu_{\e,\e,a}}"']
\arrow[r, equal]
\&
\TT_{\e}\TT_{\e}a
\arrow[d, "{\TT_{\e}u_{\TT_{\e}a}}"']
\&
\TT_{\e}\TT_{0}b
\arrow[ld, swap, "{\TT_{\e}\TT_{0}\psi}"']
\arrow[d, swap, "{\TT_{\e}\TT_{(0\leq2\zeta),b}}"']
\&
\TT_{\e}\TT_{\zeta}c
\arrow[dr, swap, "{\mu_{\e,\zeta,c}}"' description]
\arrow[d, "{\TT_{\e}\TT_{\zeta}\psi'}"']
\\
\&
\TT_{\e}\TT_{\e}a
\arrow[ul, "{\mu_{\e,\e,a}}"']
\arrow[d, "{\TT_{(\e\leq\e+2\zeta),\TT_{\e}a}}"']
\&
\TT_{\e}\TT_{0}\TT_{\e}a
\arrow[l, "{\mu_{\e,0,\TT_{\e}a}}"']
\arrow[ul, swap, "{\TT_{\e}\mu_{0,\e,a}}"']
\arrow[d, "{\TT_{\e}\TT_{(0\leq2\zeta),\TT_\e a}}"']
\&
\TT_{\e}\TT_{2\zeta}b
\arrow[dl, "{\TT_{\e}\TT_{2\zeta}\psi}"']
\&
\TT_{\e}\TT_{\zeta}\TT_{\zeta}b
\arrow[l, "{\TT_{\e}\mu_{\zeta,\zeta,b}}"']
\arrow[dl, "{\TT_{\e}\TT_{\zeta}\TT_{\zeta}\psi}"']
\arrow[d, "{\mu_{\e,\zeta,\TT_{\zeta}b}}"']
\&
\TT_{\e+\zeta}c
\arrow[dl, swap, "{\TT_{\e+\zeta}\psi'}"', description]
\ar[dddlll, bend left, "{\TT_{\e + \zeta}[\psi'']}"]
\\
\&
\TT_{\e+2\zeta}\TT_{\e}a
\arrow[ldd, "{\mu_{\e+2\zeta,\e,a}}"']
\&
\TT_{\e}\TT_{2\zeta}\TT_{\e}a
\arrow[l, "{\mu_{\e,2\zeta,\TT_{\e}a}}"']
\&
\TT_{\e}\TT_{\zeta}\TT_{\zeta}\TT_{\e}a
\arrow[l, "{\TT_{\e}\mu_{\zeta,\zeta,\TT_{\e}a}}"']
\arrow[d, "{\mu_{\e,\zeta,\TT_{\zeta}\TT_{\e}a}}"']
\&
\TT_{\e+\zeta}\TT_{\zeta}b
\arrow[dl, swap, "{\TT_{\e+\zeta}\TT_{\zeta}\psi}"' description]
\\
\&
\&
\&
\TT_{\e+\zeta}\TT_{\zeta}\TT_{\e}a
\arrow[dl, swap, "{\TT_{\e+\zeta}\mu_{\zeta,\e,a}}"']
\arrow[ull, swap, "{\mu_{\e+\zeta,\zeta,\TT_{\e}a}}"']
\\
\TT_{2(\e+\zeta)}a
\&
\&
\TT_{\e+\zeta}\TT_{\e+\zeta}a
\arrow[ll, "{\mu_{\e+\zeta,\e+\zeta,a}}"']
\end{tikzcd}
};
\end{tikzpicture}
\end{equation*}
% Note that the perimeter gives half the diagram to define an $\e + \zeta$-interleaving.
Interchanging $a$ with $c$ and an analogous argument gives the other half of the interleaving.
Thus, $(\varphi'',\psi'')$ forms an $(\e+\zeta)$-interleaving of $a$ and $c$.
Therefore, $d_{(\CC,\TT)}$ has the triangle inequality and so it defines an extended pseudometric on the objects of $\CC$.
\end{proof}

Theorem~\ref{theorem:interleavingExtPseudometric} says that every flow $\TT$ on a category $\CC$ induces an interleaving distance $d_{(\CC,\TT)}$ on $\CC$, making a symmetric Lawvere metric space $(\CC,d_{(\CC,\TT)})$.
Notice that the definition of the interleaving distance depends not only on the category $\mathcal{C}$ but also on the choice of the flow $\TT$ on $\CC$.
That means there are possibly many different interleaving distances  for the same category.
%It is easy to check that if $\TT$ and $\TT'$ are lax flows on the same category with $\TT_\e\cong\TT'_{\e}$ for each $\e\geq0$, then $d_{(\CC,\TT)}=d_{(\CC,\TT')}$.
When a particular choice of $\TT$ is implicit, we abuse notation and write $d_{\CC}$ for the interleaving distance.

%----------------------------------
\section{Examples}
\label{Sec:Examples}
%----------------------------------
In this section, we show that many commonly used metrics are actually  special cases of the interleaving distance.
For a given flow $\Omega$ on a poset $\PP$, we discuss how $\DD^{\PP}$ inherits a flow $\TT$ from $\Omega$, so that the notion of an $\Omega_{\e}$-interleaving in the context of categories of generalized persistence modules $\DD^{\PP}$ as defined in \cite{Bubenik2015} become a special case of a weak $\e$-interleaving in the context of categories with a flow.
In addition, we show that this abstract definition of the interleaving distance unifies some important distances which are commonly used in TDA and real analysis in general.
Namely we show that the Hausdorff and the $L^{\infty}$-distances are examples of interleaving distances in this general perspective.

\subsection{Interleavings on generalized persistence modules}
\label{Ssec:InterleavingsGPMs}
One of the most important tools of applied topology for the study of data is persistent homology \cite{Edelsbrunner2002,Zomorodian2004}.
The traditional presentation of persistence investigates a functor from the set of real numbers $(\R,\leq)$ viewed as a poset, to the category $\Vect_k$ of $k$-vector spaces, for some field $k$; such functors are called \textbf{persistence modules}.
Interest in defining a metric for comparison of such objects led to the original definition of the interleaving distance \cite{Chazal2009b}, a generalization of the bottleneck distance (see, e.g., \cite{Edelsbrunner2010}) which is commonly used in computational applications.
In this section, we discuss the relationship between our Definition~\ref{definition:interleavings} and the interleaving distance as previously defined.
In particular, we will follow the definition as presented in \cite{Bubenik2015} for generalized persistence modules.

Let $\PP$ be a poset.
A \textbf{translation} on  $\PP$ is an endofunctor $\Gamma:\PP\to\PP$ together with a natural transformation
$\eta:I_{\PP}\Rightarrow\Gamma$.
The collection $\mathbf{Trans}_{\PP}$ of all translations in $\PP$ forms a full subcategory of $\End(\PP)$;
in particular it is a strict monoidal category \cite[Section~5.1]{Bubenik2015}.
A \textbf{superlinear family of translations} $\Omega$, is a family of translations $\Omega_{\e}$ on $\PP$, for $\e\geq0$, such that and $I_\PP\leq\Omega_{0}$ and $\Omega_{\e}\Omega_{\zeta}\leq\Omega_{\e+\zeta}$.
As indicated in \cite[Section~5.1]{Bubenik2015}, a superlinear family of translations is simply a lax monoidal functor
\begin{equation*}
	\Omega:([0,\infty),+,0)\to\Trans_{\PP}
\end{equation*}
between strict monoidal categories.
\begin{definition}
Let $\PP$ be a poset together with a superlinear family of translations $\Omega$ and let $\DD$ be any category.
Then we call any functor $\Ffunc:\mathcal{P}\to\mathcal{D}$ a \textbf{generalized persistence module}.
We call the functor category $\mathcal{D}^{\mathcal{P}}$ a \textbf{generalized persistence module category}, or simply a \textbf{GPM-category}.
\end{definition}
Let $\Omega$ be a superlinear family of translations on $\PP$ and let $\e\geq0$.
By definition, the translation $\Omega_{\e}:\PP\to\PP$ is equipped with a natural transformation
$\eta_{\e}:I_{\PP}\Rightarrow\Omega_{\e}$.
This induces a natural transformation $\Ffunc\eta_{\e}:\Ffunc\Rightarrow \Ffunc\Omega_{\e}$.
Notice that when $\PP = (\R,\leq)$ and  $\DD = \Vect$ we have the standard persistence module framework.

\begin{definition}
[\cite{Bubenik2015}]
\label{definition:BubenikInterleaving}
Let $\Omega$ be a superlinear family of translations on a poset $\PP$, and let $\DD$ be a category.
Two generalized persistence modules $\Ffunc,\Gfunc:\PP\to\DD$ are \textbf{$\Omega_{\e}$-interleaved}
if there exist a pair of natural transformations
$\varphi:\Ffunc\Rightarrow\Gfunc\Omega_{\e}$
and
$\psi:\Gfunc\Rightarrow\Ffunc\Omega_{\e}$
such that the diagram
\begin{equation}
\label{eqn:BubenikInterleaving}
	\begin{tikzcd}
	\Ffunc
	\arrow[d, Rightarrow, "{\Ffunc\eta_\e}"']
	\&
	\Gfunc
	\arrow[dl, Rightarrow, "{\psi}"' very near start]
	\arrow[d, Rightarrow, swap, "{\Gfunc\eta_\e}"']
	\\
	\Ffunc\Omega_\e
		\arrow[d, Rightarrow,  "{\Ffunc\eta_\e\Omega_\e}"']
		\&
	\Gfunc\Omega_\e
	\arrow[Leftarrow,  ul, crossing over, "{\varphi}"' very near end]
		\arrow[d, Rightarrow, swap, "{\Gfunc\eta_\e\Omega_\e}"']
	\arrow[dl, Rightarrow,  "{\psi\Omega_\e}" very near start]
	\\
	\Ffunc\Omega_{\e}\Omega_{\e}\&
	\Gfunc\Omega_{\e}\Omega_{\e}
	\arrow[Leftarrow,  ul, crossing over, "{\phi\Omega_\e}" very near end]
	\end{tikzcd}
\end{equation}
commutes.
We call every such pair $(\varphi,\psi)$ an $\Omega_{\e}$-interleaving.

The interleaving distance with respect to $\Omega$ is
\begin{equation*}
d_{\Omega}(\Ffunc,\Gfunc)=\inf\{\e\geq0 \mid \Ffunc,\Gfunc\text{ are }\Omega_\e\text{-interleaved}\}.
\end{equation*}
If $\Ffunc$ and $\Gfunc$ are not interleaved for any $\e$, we set $d_{\Omega}(\Ffunc,\Gfunc) = \infty$.
\end{definition}

This definition gives an extended pseudometric on $\DD^{\PP}$ \cite[Theorem~3.21]{Bubenik2015}.
Notice how similar this definition is to Definition~\ref{definition:interleavings}.
In particular, in Diagram~\ref{eqn:BubenikInterleaving} the parallelograms commute by definition, so checking commutativity  splits into checking the two triangles commute.
On the other hand, Diagram~\ref{eq:interleaving} requires checking that two pentagons commute.
Essentially, the difference between the definitions comes down to working around the definition of the coherence natural transformations;
if the flow is strict and thus the coherence natural transformations are identities, the pentagon diagrams will collapse down into triangles.
We will now investigate the exact relationship between the two definitions.

\begin{lemma}
\label{lemma:SuperlinearIsLax}
Let $\PP=(\PP,\leq)$ be a poset.
Any superlinear family of translations on $\PP$ forms a flow on $\PP$ and vice versa.
\end{lemma}
\begin{proof}
Assume that $\PP$ is equipped with a superlinear family of translations
\begin{equation*}
\Omega:([0,\infty),+,0)\to\mathbf{Trans}_{\PP}.
\end{equation*}
By definition, for each $\e\geq0$, the endofunctor $\Omega_{\e}:\PP\to\PP$ is equipped with a natural transformation $\eta_\e:I_\PP\Rightarrow\Omega_\e$.
Because $\PP$ is a poset, $\eta_{\e}:I_\PP\Rightarrow\Omega_\e$ factors through $\eta_0$;
i.e.~the diagram

\begin{equation*}
\begin{tikzcd}
I_{\PP}
\arrow[r, Rightarrow, "{\eta_0}"]
\arrow[rd, Rightarrow, "{\eta_{\e}}"']
\&
\Omega_{0}
\arrow[d, Rightarrow, "{\Omega_{(0\leq\e)}}"]
\\
\& \Omega_{\e},
\end{tikzcd}
\end{equation*}
commutes.
Set $u=\eta_0$, and set $\mu_{\e,\zeta}$ to be the natural transformation induced by $\Omega_{\e}\Omega_{\zeta}\leq\Omega_{\e+\zeta}$.
Then it is easy to check that $(\Omega,u,\mu)$ forms a flow on $\PP$.

Vice versa, assume that $(\Omega,u,\mu)$ is a flow on $\PP$.
Set $\eta_{0}=u$ and $\eta_{\e}=\Omega_{(0\leq\e)}\circ u$.
Each endofunctor $\Omega_{\e}$ of $\PP$ is equipped with $\eta_\e:I_{\PP}\Rightarrow\Omega_\e$.
It is again easy to check that $\Omega$ forms a superlinear family of translations on $\PP$.
\end{proof}

Therefore a superlinear family of translations on $\PP$ is the same thing as a flow on $\PP$.
However, in order to use Definition~\ref{definition:interleavings}, we need a flow on $\DD^\PP$.
Assume we have a  flow $\Omega$ on $\PP$.
Then we can define the  flow $\TT$ on $\DD^{\PP}$ induced by $\Omega$ as a pre-composition with $\Omega$;
specifically, $ -\cdot\Omega_{\e}:\DD^{\PP}\to\DD^{\PP}$.
Moreover define the coherence natural transformations of this flow to be the pre-compositions of the coherence natural transformations defined in the proof of Lemma~\ref{lemma:SuperlinearIsLax};
specifically
$-\cdot u:I_{\DD^{\PP}}\Rightarrow-\cdot\Omega_{0}$
and
$-\cdot\mu_{\e,\zeta}:-\cdot\Omega_{\e}\Omega_{\zeta}\Rightarrow-\cdot\Omega_{\e+\zeta}$.
It is then a formality to check that the collection $\{-\cdot \Omega_\e\}_{\e\geq0}$ together with the coherence natural transformations forms a  flow on $\DD^{\PP}$, denoted by $-\cdot\Omega$.

Our next task is to investigate the relationship between Definitions~\ref{definition:interleavings} and \ref{definition:BubenikInterleaving}.

\begin{theorem}
\label{theorem:BubenikVsOurDefns}
Let  $\Omega$ be a  superlinear family of translations on $\PP$ and $-\cdot\Omega$ the induced  flow on $\DD^\PP$.
Then for any $\Ffunc,\Gfunc: \PP \to \DD$,
if $\Ffunc,\Gfunc$ are $\Omega_\e$-interleaved, then $\Ffunc,\Gfunc$ are weakly $\e$-interleaved.
This implies
\begin{equation*}
d_{(\DD^\PP, -\cdot\Omega)} (\Ffunc,\Gfunc) \leq d_{\Omega}(\Ffunc,\Gfunc).
\end{equation*}
In the case that $\Omega$ is a strict flow on $\PP$, the above is an equality.

\end{theorem}

\begin{proof}
Assume that $\Ffunc,\Gfunc$ are $\Omega_\e$-interleaved (Definition~\ref{definition:BubenikInterleaving}) and let $\eta_\e:I_\PP \Rightarrow \Omega_\e$ be the natural transformation coming from the superlinear family of translations.
Let $\TT=-\cdot\Omega$ be the induced flow on $\DD^{\PP}$ with coherence natural transformations $\hat u = -\cdot u$ and $\hat{\mu}_{\e,\zeta} = -\cdot \mu_{\e,\zeta}$.

Consider the following diagram.
\begin{equation}
\label{eq:BubenikProof}
\begin{tikzcd}
\Ffunc\Omega_{0}
\arrow[ddr, swap, Rightarrow, "{\Ffunc\Omega_{(0\leq\e)}}"']
\arrow[dddd, Rightarrow, "{\Ffunc\Omega_{(0\leq2\e)}}"']
\&
\&
\&
\Ffunc
\arrow[dd, Rightarrow, "{\Ffunc\eta_{\e}}"']
\arrow[ddrr, Rightarrow, swap, "{\varphi}"']
\arrow[lll, Rightarrow, "{\Ffunc u}"']
\\
\\
\&
\Ffunc\Omega_\e
\arrow[ddl, Rightarrow, swap,  "{\Ffunc\Omega_{(\e\leq2\e)}}"' near end]
\arrow[rr, equal]
\&
\&
\Ffunc\Omega_\e
\arrow[dd, Rightarrow, "{\Ffunc\eta_{\e}\Omega_\e}"']
\arrow[dl, Rightarrow, "{\Ffunc\eta_0\Omega_{\e}}"']
\&
\&
\Gfunc\Omega_\e
\arrow[ddll, Rightarrow, swap, "{\psi\Omega_\e}"']
\\
\&
\&
\Ffunc\Omega_0\Omega_\e
\arrow[ul, Rightarrow, swap, "{\Ffunc\mu_{0,\e}}"']
\arrow[dr, Rightarrow, "{\Ffunc\Omega_{(0\leq\e)}\Omega_\e}"']
\\
\Ffunc\Omega_{2\e}
\&
\&
\&
\Ffunc\Omega_{\e}\Omega_\e
\arrow[lll, Rightarrow, swap, "{\Ffunc\mu_{\e,\e}}"']
\end{tikzcd}
\end{equation}
The rightmost triangle commutes because $\phi$ and $\psi$ form a $\Omega_\e$-interleaving, while the rest of the cells commute by definition of a  flow (Definition~\ref{definition:catWithLaxAction}).
The perimeter of Diagram~\ref{eq:BubenikProof} gives the left half of Diagram~\ref{eq:interleaving}.
An analogous argument gives the other commuting pentagon;
thus $\Ffunc,\Gfunc$ are weakly $\e$-interleaved.

When the flow $\Omega$ on $\PP$ is strict, $\Ffunc \mu_{\e,\zeta}$ and $\Ffunc u$ are identities by definition.
Thus, a weak  $\e$-interleaving immediately induces an $\Omega_\e$-interleaving, and so the interleaving distances agree.
\end{proof}

\subsection{Interleavings on posets}
\label{Ssec:Hausdorff}
Rather than passing from a flow on a poset $\PP$ to a flow on $\DD^\PP$, we now look at the interleaving distance induced on $\PP$ itself.
Let $\PP$ be a poset together with a flow $\Omega$,
and let $d_{\PP}$ be the interleaving distance on $\PP$ induced by $\Omega$ (Definition~\ref{definition:interleavings}).
The extra structure of the poset category makes characterizing the interleaving distance rather simple, as seen in the following lemma.

\begin{lemma}
	\label{lemma:Poset}
Two objects $a,b \in \Obj\PP$ are $\e$-interleaved if and only if there exist morphisms $\varphi:a\to\Omega_{\e}b$ and $\psi:b\to\Omega_{\e}a$. So, the interleaving distance on $\PP$ induced by $\Omega$ is given by
	\begin{equation*}
	d_{\PP}(a,b)=\inf\{\e\geq0\mid \exists\; \varphi:a\to\Omega_{\e}b\text{ and }\psi:b\to\Omega_{\e}a\}.
	\end{equation*}
\end{lemma}

\begin{proof}
Let $a$ and $b$ be two objects in $\PP$ and let morphisms $\varphi:a\to\Omega_{\e}b$ and $\psi:b\to\Omega_{\e}a$ be given.
Then all morphisms of Diagram~\ref{eq:interleaving} exist and, because $\PP$ is a poset, the diagram must commute.
Thus any pair of morphisms $\phi$ and $\psi$ gives rise to an $\e$-interleaving and the lemma follows.
\end{proof}

We now show how to realize the Hausdorff distance on subsets of a metric space and the $L^{\infty}$-distance on $\R^{n}$ as interleaving distances on poset categories with flows.

\subsubsection{The Hausdorff Distance}
Fix a metric space $(\mathbb{X},d)$.
Let $S(\mathbb{X})$ be the poset category consisting of all nonempty subsets of $\mathbb{X}$ with poset given by inclusion.
Define  $A_{\e}=\cup_{a\in A}\{x\in\X\mid d(a,x)\leq\e\}$.
The \textbf{Hausdorff distance} is an extended pseudometric on $S(\X)$ given by
\begin{equation*}
	d_{H}(A,B)=\inf\{\e\geq0 \mid A\subset B_{\e}\text{ and }B\subset A_{\e}\}.
\end{equation*}

We  define a flow $\Omega$ on $S(\mathbb{X})$ as follows.
For each $\e\geq0$, define the $\e$-translation $\Omega_{\e}$ on $S(\mathbb{X})$  by  $\Omega_{\e}(A):=A_{\e}$ and  $\Omega_{\e}[A\subseteq B]$ the induced inclusion $A_{\e}\subseteq B_{\e}$.
Define the coherence natural transformations $u:I_{\PP}\Rightarrow\Omega_{0}$
and
$\mu_{\e,\zeta}:\Omega_{\e}\Omega_{\zeta}\Rightarrow\Omega_{\e+\zeta}$, $\e,\zeta\geq0$,
to be the obvious families of inclusions $u_{A}:A\subseteq A_{0}$ and $\mu_{\e,\zeta,A}:(A_{\e})_{\zeta}\subseteq A_{\e+\zeta}$, $A\subset\X$.
Naturality follows from the poset structure of $S(\X)$.
Again by the poset structure of $S(\X)$, it is easy to check that
$\Omega=(\Omega, u, \mu)$ is a flow on $S(\X)$.

\begin{theorem}
	The interleaving distance on $S(\mathbb{X})$ induced by the flow $\Omega$, coincides with the Hausdorff distance on $S(\mathbb{X})$.
Specifically, for any $A,B \in S(\X)$,
\begin{equation*}
d_{S(\X)}(A,B) = d_H(A,B)
\end{equation*}

\end{theorem}
\begin{proof}
	Clear, by definition of the Hausdorff distance and Lemma~\ref{lemma:Poset}.
\end{proof}

\subsubsection{$L^\infty$-distance on $\R^n$}
Let $\R^{n}$ be the set of all $n$-tuples of real numbers.
The $L^{\infty}$-norm on $\R^n$ is defined as follows.
Let $a=(a_1,\ldots,a_n)$ and $b=(b_1,\ldots,b_n)$ be two $n$-tuples in $\R^n$. Then define
\begin{equation*}
\|a-b\|_{\infty}=\max\{|a_i-b_i| : i=1,\ldots,n\}
\end{equation*}
We now realize this metric as an interleaving distance.
Consider $\R^n$ as the poset $(\R^{n},\leq)$ where  $a \leq b$ when $a_i\leq b_i$ for all $i=1,\ldots,n$. Define a strict flow $\Omega$ on $(\R^{n},\leq)$ as follows.
Let $\e\geq0$ and, for ease of notation, let $a + \e = (a_1+\e, \cdots, a_n + \e)$.
Define the $\e$-translation $\Omega_{\e}:(\R^{n},\leq)\to(\R^{n},\leq)$ by $\Omega_{\e}(a)=a+\e$ and $\Omega_\e[a\leq b] = (a + \e \leq b+\e)$.

Let $\e,\zeta\geq0$ and define $\Omega_{(\e\leq\zeta)} (a)= (a + \e \leq a + \zeta)$.
Clearly this collection forms a natural transformation $\Omega_{(\e\leq\zeta)}:\Omega_{\e}\Rightarrow\Omega_{\zeta}$.
We easily check that
\begin{equation*}
	\Omega:[0,\infty)\to\End(\R^{n},\leq), \,\e\mapsto\Omega_{\e}
\end{equation*}
forms a strict $[0,\infty)$-monoidal functor, i.e. a strict flow on $(\R^{n},\leq)$.
Denote the associated interleaving distance by $d_{(\R^{n},\leq)}$.
\begin{theorem}
	The interleaving distance on $\R^{n}$ induced by the strict flow $\Omega$, coincides with the $L^{\infty}$-distance on $\R^{n}$.
	That is, for any $a,b \in \R^n$,
	\begin{equation*}
	d_{(\R^{n},\leq)} (a,b) = \|a-b\|_\infty.
	\end{equation*}

\end{theorem}
\begin{proof}
	Let $a,b \in \R^n$ be two $n$-tuples.
	By Lemma~\ref{lemma:Poset},  we have
	\begin{align*}
		d_{(\R^{n},\leq)}(a,b)&=\inf\{\e\geq0\mid a\leq b+\e\text{ and }b\leq a+\e\}\\
		% &=\inf\{\e\geq0\mid a_i\leq b_i+\e\text{ and }b_i\leq a_i+\e\}\\
		&=\inf\{\e\geq0\mid b_i-\e\leq a_i\leq b_i+\e\text{ for all }i=1,\ldots,n\} \\
		&=\inf\{\e\geq0\mid \text{ }|a_i-b_i|\leq\e\text{ for all }i=1,\ldots,n\} \\
		&=\max\{|a_i-b_i| : %\text{ for all }
		i=1,\ldots,n\}\\
		&=\|a-b\|_{\infty}
	\end{align*}
as claimed.
\end{proof}

\subsection{Interleavings on slice categories}
Let $\CC$ be a category and let $c$ be an object in $\CC$.
With this data we can construct a category denoted $(\CC\downarrow c)$ called a \textbf{slice category}.
The objects in $(\CC\downarrow c)$ are tuples, $(a,f)$, where $a$ is an object in $\CC$ and $f\in\Hom_{\CC}(a,c)$.
The morphisms $\varphi:(a,f)\to(b,g)$ in $(\CC\downarrow c)$ are morphisms $\varphi:a\to b$ in $\CC$ such that $g\circ\varphi=f$.
Now we define the $L^{\infty}$-distance on the slice category $(\Top\downarrow\R)$ and then generalize to $(\Top\downarrow\M)$ for an arbitrary metric space $\M$ and realize them as examples of the interleaving distance.

\subsubsection{The $L^{\infty}$-distance on $(\Top\downarrow\R)$}
\label{Ssec:R-spaces}

Given $f:\X\to\R$ and $g:\X\to\R$ define their $L^{\infty}$-distance as
\begin{equation*}
||f-g||_{\infty}=\sup_{x\in\mathbb{X}} |f(x)-g(x)|.
\end{equation*}
Note that this definition requires the same domain for $f$ and $g$.
We can extend the definition of the $L^{\infty}$-distance to arbitrary $\R$-valued functions.
Consider the slice category $(\Top\downarrow\R)$, whose objects are pairs $(\X,f)$ consisting of a topological space $\X$ and an $\R$-valued function $f$ on $\X$, called \textbf{$\R$-spaces}.
A morphism $\phi:(\X,f)\to(\Y,g)$ in $(\Top\downarrow\R)$ is a continuous map $\phi:\X\to\Y$ such that $\phi\circ g=f$, called a \textbf{function preserving map}.

\begin{definition}
For general $\R$-valued functions $f:\X\to\R$ and $g:\Y\to\R$, the $L^{\infty}$-distance is defined to be
$$d_{\infty}((\X,f),(\Y,g))=\inf_{\Phi:\X\to\Y}||f-g\circ\Phi||_{\infty}$$
where $\Phi$ runs over all homeomorphisms.
If $\X$ and $\Y$ are not homeomorphic, we set  $d_{\infty}(f,g)=\infty$.
\end{definition}
It is immediate that for $f$ and $g$ defined on the same domain, $d_\infty(f,g) \leq \|f-g\|_\infty$.

We now realize $d_\infty$ as an interleaving distance on $(\Top\downarrow\R)$ by defining a flow $\TT$ on $(\Top\downarrow\R)$ as follows.
% Given $\e \geq 0$,  the $\e$-smoothing functor $\mathcal{T}_{\e}:(\Top\downarrow\R)\to(\Top\downarrow\R)$ is defined as follows.
\begin{itemize}
\item
Let $\e\geq0$.
For $(\X,f) \in \Obj(\Top \downarrow \R)$, $\TT_{\e}(\X,f)=(\X_{\e},f_{\e})$ where
$\X_{\e}=\X\times[-\e,\e]$ and $f_{\e}(x,t)=f(x)+t$.
\item
For a morphism $\phi:(\X,f)\to(\Y,g)$, let $\TT_{\e}[\phi]:(\X_{\e},f_{\e})\to(\Y_{\e},g_{\e})$, $(x,t)\mapsto (\phi(x),t)$.
\item
Let $0\leq\e\leq\zeta$.
Define $\TT_{(\e\leq\zeta)}:\TT_{\e}\Rightarrow\TT_{\zeta}$ to be the natural transformation $(\X_{\e},f_{\e})\to (\X_{\zeta},f_{\zeta})$, $(x,t)\mapsto (x,t)$.
\item
For $(\X,f) \in \Obj(\Top \downarrow \R)$, take
 $u_f:(\X,f)\to(\X_{0},f_{0})$, $x\mapsto(x,0)$.
\item
Let $\e,\zeta\geq0$.
For $(\X,f) \in \Obj(\Top \downarrow \R)$, consider
$\mu_{\e,\zeta,f}:((\X_{\zeta})_{\e},(f_{\zeta})_{\e})\to(\X_{\e+\zeta},f_{\e+\zeta})$, $((x,t),s)\mapsto(x,t+s)$.
\end{itemize}
It is largely bookkeeping to check that $\TT=(\TT,u,\mu)$ is a flow on $(\Top\downarrow\R)$.
%Denote the  associated interleaving distance of the category with a flow $((\Top\downarrow\R),\TT)$ by
%$d_{(\Top\downarrow\R)}$.

\begin{theorem}
\label{theorem:LInftyForSlice}
The interleaving distance $d_{((\Top\downarrow\R),\TT)}$ coincides with the distance $d_{\infty}$.
That is, given $(\X,f), (\Y,g)$ in $(\Top \downarrow \R)$,
\begin{equation*}
d_{((\Top\downarrow\R),\TT)}((\X,f), (\Y,g)) = d_{\infty}((\X,f), (\Y,g)).
\end{equation*}

\end{theorem}
\begin{proof}
Let $\e\geq0$.
It suffices to show that $(\mathbb{X},f)$ and $(\mathbb{Y},g)$ are weakly $\e$-interleaved if and only if for some homeomorphism $\Phi:\mathbb{X}\to\mathbb{Y}$, $|f(x)-g(\Phi(x))|\leq\e$ for all $x\in\mathbb{X}$.

Suppose that  $(\mathbb{X},f)$ and $(\mathbb{Y},g)$ are weakly $\e$-interleaved via a pair of morphisms $\varphi:(\mathbb{X},f)\to\TT_{\e}(\mathbb{Y},g)$ and  $\psi:(\mathbb{Y},g)\to\TT_{\e}(\mathbb{X},f)$.
Let $\Phi=p_{1}\circ\varphi$ and $\Psi=p_{1}\circ\psi$ be the projections of $\varphi$ and $\psi$ to the first coordinate. Since the morphisms $\varphi:(\mathbb{X},f)\to\TT_{\e}(\mathbb{Y},g)$ and  $\psi:(\mathbb{Y},g)\to\TT_{\e}(\mathbb{X},f)$ are function preserving maps, we can write
\begin{align*}
\varphi(x)&=\big(\Phi(x),f(x)-g\Phi(x)\big) \\
\psi(y) & =\big(\Psi(y),g(y)-f\Psi(y)\big).
\end{align*}
As $(\varphi,\psi)$ is a weak $\e$-interleaving, we can chase an element around the left pentagon of the commutative diagram of Equation~\ref{eq:interleaving} to see that
\begin{equation*}
(x,0) = (\Psi\Phi(x),f(x) - f(\Psi\Phi(x))).
\end{equation*}
So, $\Psi\Phi = I_\X$ and a similar argument gives that $\Phi\Psi=I_\Y$.
Therefore, $\Phi$ is a homeomorphism.
Because $\varphi:(\mathbb{X},f)\to\TT_{\e}(\mathbb{Y},g)$ is function preserving map, the map $\Phi$ further satisfies $|f(x)-g(\Phi(x))|\leq\e$ for all $x\in\mathbb{X}$.

Now, assume that there exists a homeomorphism $\Phi:\X \to \Y$ such that $|f(x)-g(\Phi(x))|\leq\e$ for all $x\in\mathbb{X}$ and define $\Psi = \Phi \inv$.
Define the morphisms $\varphi:(\X,f) \to (\Y_\e,g_\e)$ and $\psi:(\Y,g) \to (\X_\e,f_\e)$ by the formulas
\begin{align*}
\varphi(x)&=\big(\Phi(x),f(x)-g\Phi(x)\big) \text{ and}\\
\psi(y)&=\big(\Psi(y),g(y)-f\Psi(y)\big).
\end{align*}
It is easy to check that these are function preserving maps.
Again, diagram chasing shows that the diagram of Equation~\ref{eq:interleaving} commutes, so $\phi$ and $\psi$ form a weak $\e$-interleaving of $(\mathbb{X},f)$ and $(\mathbb{Y},g)$.
\end{proof}
This example shows that one needs to work the coherence natural transformations in Definition~\ref{definition:interleavings} to define weak $\e$-interleavings rather than working with a strict analogue of Definition~\ref{definition:BubenikInterleaving}.
That is to say, a definition which checks if the pentagons in Diagram~\ref{eq:interleaving} commute rather than just triangles as in Equation~\ref{definition:BubenikInterleaving}.
With the choice of category of $\R$-spaces as a slice category, Theorem~\ref{theorem:LInftyForSlice} cannot be proven considering the condition on the interleaving relation of Equation~\ref{eq:interleaving} that triangles commute for the definition of the $\e$-interleaving relation because only after composing with the coherence natural transformations $\mu_{\e,\e}$, can the point
\begin{equation*}
(\Psi\Phi(x),f(x)-g(\Phi(x)),g(\Phi(x))-f(\Psi\Phi(x)))
\end{equation*}
be identified with $(x,0)$.

In \cite[Remark~5.1]{Lesnick2015}, the author discusses a similar result to Theorem~\ref{theorem:LInftyForSlice}.
However, he considers the category of $\R$-spaces with a larger collection of morphisms than those in the slice category.
This relaxation in the category means the flow is strict, and thus the interleaving relation consists of triangles rather than pentagons.
Working with these slice categories as defined has the added benefit that
it is easier to define both the sublevel and level set  filtration functors, which in turn, provides us with the stability of both sublevel set and level set persistent homology as will be discussed in Section~\ref{ssec:persistence-stability}.

\subsubsection{The $L^{\infty}$-distance on $(\Top\downarrow\M)$}
\label{Ssec:M-spaces}
We now extend the $d_\infty$ distance to arbitrary metric spaces.
Fix a metric space $(\M,d)$.
\begin{definition}
	For general $\M$-valued functions $f:\X\to\M$ and $g:\Y\to\M$, the $L^{\infty}$-distance is defined to be
	$$d_{\infty}(f,g)=\inf_{\Phi:\mathbb{X}\to\mathbb{Y}}\sup_{x\in\X}d(f(x),g(\Phi(x)))$$
	where $\Phi$ runs over all homeomorphisms.
	If $\X$ and $\Y$ are not homeomorphic, we set  $d_{\infty}(f,g)=\infty$.
\end{definition}
We now show how $d_\infty$ can be realized as an interleaving distance on $(\Top\downarrow\M)$.
We define a  flow $\widehat{\TT}$ on $(\Top\downarrow\M)$, which is very similar to the flow $\TT$ as defined in Section~\ref{Ssec:R-spaces}, as follows.
%Given $\e\geq 0$, the $\e$-translation $\widehat{\TT}_{\e}:(\Top\downarrow\M)\to(\Top\downarrow\M)$ is defined as follows.
\begin{itemize}
\item
For $(\mathbb{X},f) \in \Obj(\Top \downarrow \M)$, $\widehat{\TT}_{\e}(\mathbb{X},f)=(\QQ_{\e}(\mathbb{X},f),p_2)$ where
$\QQ_{\e}(\mathbb{X},f):=\{(x,m)\in\X\times\M\mid d(f(x),m)\leq\e\}$ and $p_2$ is the projection to the second coordinate.
\item
For a morphism $\phi:(\X,f)\to(\Y,g)$, let $\widehat{\TT}_{\e}[\phi]:(\QQ_{\e}(\X,f),p_2)\to(\QQ_{\e}(\Y,g),p_2)$, $(x,m)\mapsto (\phi(x),m)$.
%It is easy to check that $\QQ_{0}(\X,f)=\{(x,f(x))\mid x\in\X\}$; i.e.~the graph of $f$.
\item
Let $0\leq\e\leq\zeta$.
Define $\widehat{\TT}_{(\e\leq\zeta)}:\widehat{\TT}_{\e}\Rightarrow\widehat{\TT}_{\zeta}$ to be the natural transformation $(\QQ_{\e}(\X,f),p_2)\to (\QQ_{\zeta}(\X,f),p_2)$, $(x,m)\mapsto (x,m)$.
%Then  $\widehat{\TT}:[0,\infty)\to\End(\Top\downarrow\M)$, $\e\mapsto\widehat{\TT}_{\e}$ is functorial.
\item
For $(\X,f) \in \Obj(\Top \downarrow \M)$, take
$\widehat{u}_f:(\X,f)\to(\QQ_{0}(\X,f),p_{2})$, $x\mapsto(x,f(x))$.
\item
Let $\e,\zeta\geq0$.
For $(\X,f) \in \Obj(\Top \downarrow \M)$, consider
$\widehat{\mu}_{\e,\zeta,f}:(\QQ_{\e}(\QQ_{\zeta}(\X,f),p_2),p_2)\to(\QQ_{\e+\zeta}(\mathbb{X},f),p_2)$, $((x,t),s)\mapsto(x,s)$.
\end{itemize}
Again it is a formality to check that $(\widehat{\TT},\widehat{u},\widehat{\mu})$ is a flow on $(\Top\downarrow\M)$.
%Denote the associated interleaving distance on $(\Top\downarrow\M)$ by $d_{((\Top\downarrow\M),\widehat{\TT})}$.

\begin{theorem}
\label{theorem:SliceGeneral}
If $(\M,d)=(\R,||\cdot||_{\infty})$, then
	\begin{equation*}
		((\Top\downarrow\R),\widehat{\TT})\cong((\Top\downarrow\R),\TT).
	\end{equation*}
	That is, for every $\e\geq0$, $\widehat{\TT}_{\e}$ is naturally isomorphic to $\TT_\e$ as defined in Section~\ref{Ssec:R-spaces}.
	\end{theorem}

\begin{proof}
Let $\e\geq0$ and let $(\X,f)$ be an $\M$-space.
We claim that $\widehat{\TT}_{\e}(\X,f)\cong\TT_{\e}(\X,f)$.
Consider the function preserving maps $(\QQ_{\e}(\X,f),p_2)\to(\X_\e,f_\e)$, $(x,m)\mapsto (x,m-f(x))$ and $(\X_\e,f_\e)\to(\QQ_{\e}(\X,f),p_2)$, $(x,t)\mapsto (x,t+f(x))$.
Then it is easy to check that they are inverses, thus we get an isomorphism as desired.
	\end{proof}

\begin{theorem}
The interleaving distance $d_{((\Top\downarrow\M),\widehat{\TT})}$ coincides with the distance $d_{\infty}$.
That is, given $(\X,f), (\Y,g)$ in $(\Top \downarrow \M)$,
\begin{equation*}
	d_{((\Top\downarrow\M),\widehat{\TT})}((\X,f), (\Y,g)) = d_\infty((\X,f), (\Y,g)).
\end{equation*}
\end{theorem}

\begin{proof}
This proof proceeds in exactly the same manner as that of Theorem~\ref{theorem:LInftyForSlice}.
\end{proof}

%-----------------------------------------
\section{Stability Theorems}
\label{Sec:Stability}
%-----------------------------------------

So far, we have seen that the interleaving distance gives a very general framework for defining a distance on a category, and that this framework encompasses many commonly used metrics.
The next goal is to understand how these metrics relate to each other.
In particular, in this section we define colax $[0,\infty)$-equivariant functors of categories with a flow and show that they are 1-Lipschitz (non-expansive) to the respective interleaving distances.
We then show that some important 1-Lipschitz maps known from TDA can be realized as colax $[0,\infty)$-equivariant functors of categories with a flow, thus giving alternative proofs to integral stability results.

\subsection{Stability of $[0,\infty)$-equivariant functors}
% \subsection{Definition}
Firstly, we define maps between categories with a flow.
\begin{definition}
	\label{definition:equivariant}
A \textbf{colax $[0,\infty)$-equivariant functor} $\HH:\CC\to\DD$ of categories with a flow $\CC=(\CC,\TT,u,\mu)$ and $\DD=(\DD,\SS,v,\lambda)$ is an ordinary functor $\HH:\CC\to\DD$ together with, for each $\e\geq0$, a natural transformation $\eta_{\e}:\mathcal{H}\mathcal{T}_{\e}\Rightarrow\mathcal{S}_{\e}\mathcal{H}$
such that the diagrams
\begin{equation*}
\begin{array}{c}
\begin{tikzcd}
\&
\HH
\arrow[ld, swap, Rightarrow, "{I_{\HH} u}"]
\arrow[rd,  Rightarrow, "{v I_{\HH}}"]
\\
\HH\TT_{0}
\arrow[rr, Rightarrow, "{\eta_{0}}"]
\& \&
\SS_{0}\HH
\end{tikzcd}
% \\
\begin{tikzcd}
\HH\TT_{\e}
\arrow[r, Rightarrow, "{\eta_{\e}}"]
\arrow[d, swap, Rightarrow, "{I_{\HH}\TT_{(\e\leq \zeta)}}"]
\&
\SS_{\e}\HH \arrow[d, Rightarrow, "{\SS_{(\e\leq \zeta)} I_{\HH}}"]
\\
\HH\TT_{\zeta} \arrow[r,  Rightarrow, "{\eta_{\zeta}}"]
\&
\SS_{\zeta}\HH
\end{tikzcd}
\\
\begin{tikzcd}
\HH\TT_{\e}\TT_{\zeta}
\arrow[r, Rightarrow, "{\eta_{\e} I_{\TT_{\zeta}}}"]
\arrow[d, swap, Rightarrow, "{I_{\HH}\mu_{\e,\zeta}}"]
\&
\SS_{\e}\HH\TT_{\zeta}
\arrow[r, Rightarrow, "{I_{\SS_{\e}}\eta_{\zeta}}"]
\&
\SS_{\e}\SS_{\zeta}\HH
\arrow[d, Rightarrow, "{\lambda_{\e,\zeta} I_{\HH}}"]
\\
\HH\TT_{\e+\zeta}\arrow[rr,  Rightarrow, "{\eta_{\e+\zeta}}"]
\&\&
\SS_{\e+\zeta}\HH
\end{tikzcd}
\end{array}
\end{equation*}
commute for all $\e,\zeta\geq0$.
\end{definition}
If all $\eta_{\e}$ are identities then $\HH$ is called a \textbf{strict $[0,\infty)$-equivariant functor}.
If the categories $(\CC,\TT),(\DD,\SS)$ have strict flows, then a strict $[0,\infty)$-equivariant functor $\HH:(\CC,\TT)\to(\DD,\SS)$ is simply a functor $\HH:\CC\to\DD$ such that $\HH\TT_{\e}=\SS_{\e}\HH$ for all $\e\geq0$.
If in particular $\HH:\CC\to\DD$ is an equivalence of categories, then $\HH$ is called a \textbf{strict $[0,\infty)$-equivariant equivalence}.

With these definitions in hand, we now turn to understanding the relationship between the different metrics under a colax $[0,\infty)$-equivariant functor.

\begin{theorem}
\label{theorem:stability}
Let $\HH:\CC\to\DD$ be a colax $[0,\infty)$-equivariant functor of categories with a flow $\CC=(\CC,\TT,u,\mu)$ and $\DD=(\DD,\SS,v,\lambda)$.
Then, $\mathcal{H}:\CC\to\DD$ is 1-Lipschitz to the respective interleaving distances, i.e.
\begin{equation*}
d_{(\DD,\SS)}(\HH a,\HH b)\leq d_{(\CC,\TT)}(a,b)
\end{equation*}
for every pair of objects $a,b$ in $\CC$.
\end{theorem}

\begin{proof}
To show that $\mathcal{H}$ is 1-Lipschitz we have to show that $\HH$ sends weak $\e$-interleavings in $\CC$ to weak $\e$-interleavings in $\DD$.
Specifically, let $a,b$ be two objects in $\CC$ which are weakly $\e$-interleaved;
we will show that $\HH a,\HH b$ are also weakly $\e$-interleaved.

Since $a,b$ are weakly $\e$-interleaved, there exists a pair of morphisms $\varphi:a\to\TT_{\e}b$ and $\psi:b\to\TT_{\e}a$ in $\CC$, such that the diagram of Equation~\ref{eq:interleaving} commutes.

Consider the following diagram.
\begin{equation}
\label{eq:H-proof-big-dgm}
\begin{tikzcd}%[row sep=huge, column sep=huge]
\SS_{0}\HH a
\arrow[dddd, "{\TT_{(0\leq2\e),\HH a}}"']
\&
\&
\HH a
\arrow[ll, "{v_{\HH a}}"']
\arrow[ld, swap, "{\HH u_{a}}"']
\arrow[rd, swap, "{\HH \varphi}"']
\\
\&
\HH\TT_{0}a
\arrow[lu,"{\eta_{0,a}}"]
\arrow[d, "{\HH \TT_{(0\leq 2\e),a}}"']
\&
\&
\HH\TT_{\e}b
\arrow[dl, swap,"{\HH\TT_{\e}\psi}"]
\arrow[dr, "{\eta_{\e,b}}"]
\\
\&
\HH\TT_{2\e}a
\arrow[ldd, "{\eta_{2\e,a}}"' ]
\&
\HH\TT_{\e}\TT_{\e}a
\arrow[l, "{\HH\mu_{\e,\e,a}}"' ]
\arrow[rd, swap, "{\eta_{\e,\TT_{\e}a}}"' ]
\&
\&
\SS_{\e}\HH b
\arrow[dl,"{\SS_{\e}\HH\psi}"]
\\
\&
\&
\&
\SS_{\e}\HH \TT_{\e}a
\arrow[dl,"{\SS_{\e}\eta_{\e,a}}"]
\\
\SS_{2\e}\HH a
\&
\&
\SS_{\e}\SS_{\e}\HH a
\arrow[ll, "{\lambda_{\e,\e,\HH a}}"' ]
\end{tikzcd}
\end{equation}
Proceeding clockwise numbering the triangle as 1, the second region commutes as it is the functor $\HH$ applied to the left side of Equation~\ref{eq:interleaving},
the third region commutes by definition of $\eta_\e$ as a natural transformation, the fifth region commutes by the square diagram of Definition~\ref{definition:equivariant}, and the first and fourth regions commute by the additional commutative diagrams of Definition~\ref{definition:equivariant}.

Define $\Phi=\eta_{\e,b}\circ\HH\varphi$ and $\Psi=\eta_{\e,a}\circ\mathcal{H}\psi$.
Since Diagram~\ref{eq:H-proof-big-dgm} commutes, we get
$$\lambda_{\e,\e,\HH a}\circ\SS_{\e}\Psi\circ\Phi=\TT_{(0\leq2\e),\HH a}\circ v_{\HH a}.$$
Similarly we can construct an analogous diagram to show that
$$\lambda_{\e,\e,\HH b}\circ\SS_{\e}\Phi\circ\Psi=\TT_{(0\leq2\e),\HH b}\circ v_{\HH b}.$$
Therefore $(\Phi,\Psi)$ forms a weak $\e$-interleaving of $\HH a,\HH b$.
\end{proof}

\subsection{Stability of post-compositions between GPM-categories}
\label{ssec:GPM-stability}
Let $(\PP,\Omega,u,\mu)$ be a poset with a  flow (equivalently a superlinear family by Lemma~\ref{lemma:SuperlinearIsLax}) and let $(\DD^{\PP},\TT)$ and $(\EE^{\PP},\SS)$ be two GPM-categories with  flows induced by the flow $\Omega$ on the poset $\PP$.
Namely, $\TT_\e = - \cdot \Omega_\e$ and $\SS_\e=-\cdot\Omega_\e$ are the $\e$-translations on $\DD^\PP$ and $\EE^\PP$ respectively.
Now let $\Hfunc: \DD \to \EE$ be a functor and consider the post-composition functor $\Hfunc\cdot-:\DD^{\PP}\to\EE^{\PP}$, that sends each functor $\Ffunc:\PP\to\DD$ to the functor $\Hfunc\Ffunc:\PP\to\EE$.

In \cite{Bubenik2015}, the authors investigate the Lipschitz properties of such functors $\Hfunc$ with respect to interleavings  in the sense of Definition~\ref{definition:BubenikInterleaving}.
Here, we show that these functors fit into our flow framework, and
we see how this is connected to their result.

\begin{theorem}
\label{theorem:postCompIsOplax}
The post composition functor $\HH=\Hfunc\cdot-$ is colax $[0,\infty)$-equivariant.
\end{theorem}

\begin{proof}
Clearly $\HH$ is a functor, so we need to construct a natural transformation $\eta_{\e}:\HH\TT_{\e}\Rightarrow\SS_{\e}\HH$.
This is just the identity transformation: $\eta_{\e,\Ffunc}:\Hfunc(\Ffunc\Omega_{\e})=(\Hfunc\Ffunc)\Omega_{\e}$, whose naturality follows from the associativity property of composition of functors.
It is easy to check using naturality that $\HH$ together with $\eta_\e$, for $\e\geq0$, satisfies the required commutative diagrams of Definition~\ref{definition:equivariant}.
\end{proof}
\begin{corollary}
Let $\Ffunc, \Gfunc \in \DD^\PP$ and $\Hfunc: \DD \to \EE$.
Then
\begin{equation*}
d_{(\DD^\PP, -\cdot\Omega)}(\Hfunc \Ffunc, \Hfunc \Gfunc) \leq d_{(\EE^\PP, -\cdot \Omega)}(\Ffunc,\Gfunc).
\end{equation*}

\end{corollary}
\begin{proof}
The proof follows from combining Theorem~\ref{theorem:postCompIsOplax} with Theorem~\ref{theorem:stability}.
\end{proof}
We note that this corollary is closely related to the following theorem from \cite{Bubenik2015}.
\begin{theorem}
[Bubenik et al.~\cite{Bubenik2015}; Theorem~3.16]
\label{theorem:BubenikStability}
Let $\Ffunc, \Gfunc \in \DD^\PP$ and $\Hfunc: \DD \to \EE$.
Then
\begin{equation*}
d_\Omega(\Hfunc \Ffunc, \Hfunc \Gfunc) \leq d_\Omega(\Ffunc,\Gfunc).
\end{equation*}
\end{theorem}
In particular, if the flow $\Omega$ on $\PP$ is strict, Theorem~\ref{theorem:BubenikStability} is also a corollary of Theorem~\ref{theorem:postCompIsOplax}.

%----------
\subsection{Stability of persistent homology}
\label{ssec:persistence-stability}
In this section, we show that some of the key stability results of persistent homology are special cases of Theorem~\ref{theorem:stability}.
(These are the `soft' stability theorems, in the language of Bubenik et al.~\cite{Bubenik2015}.
The stability of persistence diagrams requires other technology.)

\subsubsection{$L^{\infty}$-stability of sublevel set persistent homology}
Let $((\mathbf{Top}\downarrow\R),\TT)$ be the category with a flow of $\R$-spaces with the flow $\TT$ as defined in Section~\ref{Ssec:R-spaces}, notably $\TT_\e(\X,f) = (\X_\e,f_\e)$.
Let $(\Top^{(\R,\leq)},-\cdot\Omega)$ be the GPM-category of one dimensional persistence modules with the strict flow $-\cdot\Omega$ which is induced (as a pre-composition) by the $\e$-shift functors $\Omega_{\e}:a\mapsto a+\e$ on the poset $(\R,\leq)$.
Define the \textbf{sublevel set filtration} functor $\SS:(\mathbf{Top}\downarrow\R)\to\Top^{(\R,\leq)}$ which sends an $\R$-space $(\mathbb{X},f)$ to its sublevel set filtration $\SS(\mathbb{X},f)=f^{-1}$: $a\mapsto f^{-1}(-\infty,a]$.

\begin{theorem}
	\label{theorem:SubLevelsetIsLax}
	The sublevel set filtration functor $\SS$ is colax $[0,\infty)$-equivariant.
\end{theorem}

\begin{proof}
	Firstly we need to construct for each $\e\geq0$, a natural transformation $\eta_{\e}:\SS\TT_{\e}\Rightarrow(-\cdot\Omega_{\e})\SS$,  where $\eta_{\e,f}:f_{\e}^{-1}\mapsto f^{-1}\Omega_{\e}$.
	Define $\eta_{\e,f}=(\eta_{\e,f,a})_{a\geq0}$, where $\eta_{\e,f,a}:=p_1:f_\e\inv(-\infty,a] \to f\inv(-\infty,a+\e]$ is the projection of $f_{\e}^{-1}(-\infty,a] \subseteq \X \times [-\e,\e]$ to the first coordinate.
	Then we check that the diagram
	\begin{equation*}
		\begin{array}{c}
			\begin{tikzcd}
				f_{\e}^{-1}(-\infty,a]
				\arrow[r, "{p_1}"]
				\arrow[d, "{\phi_{\e}}"']
				\&
				f^{-1}(-\infty,a+\e]
				\arrow[d,  "{\phi}"]
				\\
				g_{\e}^{-1}(-\infty,a]
				\arrow[r,  "{p_1}"]
				\&
				g^{-1}(-\infty,a+\e]
			\end{tikzcd}
			\begin{tikzcd}
				(x,t)
				\arrow[r, mapsto]
				\arrow[d, mapsto]
				\&
				x
				\arrow[d, mapsto]
				\\
				(\varphi(x),t)
				\arrow[r, mapsto]
				\&
				\varphi(x)
			\end{tikzcd}
		\end{array}
	\end{equation*}
	commutes for any morphism $\phi$ in the slice category.
	Thus the collection  $\eta_{\e} = (\eta_{\e,f})_{f}$ forms a natural transformation, i.e the diagram
	\begin{equation*}
		\begin{tikzcd}
			f_{\e}^{-1}
			\arrow[r, Rightarrow, "{\eta_{\e,f}}"]
			\arrow[d, Rightarrow, "{\mathcal{T}_{\e}[\varphi]\circ-}"']
			\&
			f^{-1}\Omega_{\e}
			\arrow[d, Rightarrow, "{\mathcal{S}_{\e}[\varphi\circ-]}"]
			\\
			g_{\e}^{-1}
			\arrow[r, Rightarrow, "{\eta_{\e,g}}"]
			\&
			g^{-1}\Omega_{\e}
		\end{tikzcd}
	\end{equation*}
	commutes.
	Hence $\eta_{\e}:\SS\TT_{\e}\Rightarrow(-\cdot\Omega_{\e})\SS$.
	It is then a formality to check that the sublevel set filtration functor $\SS$ together with the natural transformations $\eta_\e$ satisfy the axioms of a colax $[0,\infty)$-equivariant functor.
\end{proof}
Let $k$ be  field and $p$ a nonnegative integer.
Consider the $p$-dimensional homology functor (with coefficients in $k$), $\HH_{p}:\Top\to\Vect_{k}$.
We define the \textbf{$p$-dimensional sublevel set persistent homology functor} $\HH_p\SS$ as the composition
\begin{equation*}
	\begin{tikzcd}
		((\Top\downarrow\R),\TT)
		\arrow[r, "{\SS}"]
		\arrow[rr, bend left, "{\HH_p\SS}"]
		\&
		(\Top^{(\R,\leq)},-\cdot\Omega)
		\arrow[r, "{\HH_{p}\cdot-}"]
		\&
		(\Vect_{k}^{(\R,\leq)}-\cdot\Omega).
	\end{tikzcd}
\end{equation*}

\begin{lemma}
The $p$-dimensional sublevel set persistent homology functor $\HH_p\SS$ is colax $[0,\infty)$-equivariant.
\label{corollary:SPH}
\end{lemma}

\begin{proof}
Post-composing with the $p$-dimensional homology functor $\HH_{p}$, is a functor of GPM-categories with the same domain poset $(\R,\leq)$, and so, by Theorem~\ref{theorem:BubenikVsOurDefns}, the functor $\HH_{p}\cdot-$ forms an colax $[0,\infty)$-equivariant functor of categories with a flow.
Also the sublevel set filtration functor $\SS$ is colax $[0,\infty)$-equivariant.
Hence, their composition $\HH_p\SS$ is a colax $[0,\infty)$-equivariant functor of categories with a flow.
\end{proof}

\begin{corollary}[Cohein-Steiner et al.~\cite{Cohen-Steiner2007}]
Sublevel set persistence $\HH_p\SS$ is 1-Lipschitz with respect to the interleaving distance, i.e.
 \begin{equation*}
 	d_{\Vect_{k}^{(\R,\leq)}}(\HH_p\SS(\X),\HH_p\SS(\Y))\leq d_{\infty}(\X,\Y).
 \end{equation*}
\end{corollary}
\begin{proof}
The proof follows directly from Lemma~\ref{corollary:SPH} and Theorem~\ref{theorem:stability}.
\end{proof}

\subsubsection{$L^{\infty}$-stability of level set persistent homology}
Let $(\M,d)$ be a metric space.
Let $((\Top\downarrow\M),\widehat{\TT})$ be the slice category of $\M$-spaces with the flow $\widehat{\TT}$ as defined in Section~\ref{Ssec:M-spaces}.
Let $\Open(\M)$ be the poset whose objects are open subsets $\UU$ of $\M$ and morphisms are inclusions.
Consider the $\e$-translations on $\Open(\M)$, $\widehat{\Omega}_{\e}:\Open(\M)\to\Open(\M)$, $\UU\mapsto \UU^{\e}$, where
\begin{equation*}
\UU^{\e}:=\cup_{y\in \UU}\{m\in\M\mid d(m,y)\leq\e\}
\end{equation*}
We easily check that the collection $\widehat{\Omega}$ of all $\e$-translations $\widehat{\Omega}_{\e}$ on $\Open(\M)$ forms a flow on the poset $\Open(\M)$.
Let $\Top^{\Open(\M)}$ be the GPM-category of pre-cosheaves on $\M$.
Then, $\widehat{\Omega}$ induces a flow $-\cdot\widehat{\Omega}$ on $\Top^{\Open(\M)}$ (induced by pre-composition).
Consider the \textbf{level set filtration functor} $\LL:(\mathbf{Top}\downarrow\M)\to\Top^{\Open(\M)}$ that sends an $\M$-space $(\mathbb{X},f)$ to its level set filtration
\begin{equation*}
\LL(\mathbb{X},f):=f^{-1}: \UU\mapsto f^{-1}(\UU).
\end{equation*}

\begin{theorem}
\label{theorem:LevelsetIsLax}
The level set filtration functor $\LL$ is colax $[0,\infty)$-equivariant.
\end{theorem}
\begin{proof}
Firstly we need to construct a natural transformation $\eta_{\e}:\LL\widehat{\TT}_{\e}\Rightarrow(-\cdot\widehat{\Omega}_{\e})\LL$, where $\eta_{\e,f}:\LL\widehat{\TT}_{\e}(\X,f)(\UU)\mapsto f^{-1}\widehat{\Omega}_{\e}(\UU)$.
That is, to construct for each $\M$-space, $(\X,f)$ a family of maps $\eta_{\e,f,\UU}:p_2^{-1}(\UU)\mapsto f^{-1}(\UU^{\e})$ natural for all $\UU$ and for all $(\X,f)$.
Same as before define $\eta_{\e,f,\UU}$ to be the restriction to the projection on the first coordinate.
%Then Define $\eta_{\e,f}=(\eta_{\e,f,\UU})_{\UU}$, where $\eta_{\e,f,\UU}:=p_1|$.
Then we easily check that the diagram
\begin{equation*}
	\begin{array}{c}
		\begin{tikzcd}
			p_{2}^{-1}(\UU)
			\arrow[r, "{p_1}"]
			\arrow[d, "{\widehat{\TT}_{\e}[\varphi]}"']
			\&
			f^{-1}(\UU^\e)
			\arrow[d,  "{\varphi}"]
			\\
			p_2^{-1}(\UU)
			\arrow[r,  "{p_1}"]
			\&
			g^{-1}(\UU^\e)
		\end{tikzcd}
		\begin{tikzcd}
			(x,m)
			\arrow[r, mapsto]
			\arrow[d, mapsto]
			\&
			x
			\arrow[d, mapsto]
			\\
			(\varphi(x),m)
			\arrow[r, mapsto]
			\&
			\varphi(x)
		\end{tikzcd}
	\end{array}
\end{equation*}
commutes.
Thus the collection  $\eta_{\e} = (\eta_{\e,f})_{f}$ forms a natural transformation, i.e.~the diagram
\begin{equation*}
\begin{tikzcd}
f_{\e}^{-1}
\arrow[r, Rightarrow, "{\eta_{\e,f}}"]
\arrow[d, Rightarrow, "{\widehat{\TT}_{\e}[\varphi]\circ-}"']
\&
f^{-1}\widehat{\Omega}_{\e}
\arrow[d, Rightarrow, "{\varphi\widehat{\Omega}_{\e}}"]
\\
g_{\e}^{-1}
\arrow[r, Rightarrow, "{\eta_{\e,g}}"]
\&
g^{-1}\widehat{\Omega}_{\e}
\end{tikzcd}
\end{equation*}
commutes.
Hence $\eta_{\e}:\LL\widehat{\TT}_{\e}\Rightarrow(-\cdot\widehat{\Omega}_{\e})\LL$.
Then it is a formality to check that the level set filtration functor $\LL$ together with $\eta_\e$ satisfies the axioms of a colax $[0,\infty)$-equivariant functor.
\end{proof}

Let $k$ be  field and $p$ a nonnegative integer.
Consider the $p$-dimensional homology functor (with coefficients in $k$), $\HH_{p}:\Top\to\Vect_{k}$.
We define the \textbf{$p$-dimensional level set persistent homology $\HH_{p}\LL$} as the composition
\begin{equation*}
\begin{tikzcd}
((\Top\downarrow\M),\widehat{\TT})
\arrow[r, "{\LL}"]
\arrow[rr, bend left, "{\HH_{p}\LL}"]
\&
(\Top^{\Open(\M)},-\cdot\widehat{\Omega})
\arrow[r, "{\HH_{p}\cdot-}"]
\&
(\Vect_{k}^{\Open(\M)},-\cdot\widehat{\Omega}).
\end{tikzcd}
\end{equation*}

\begin{lemma}
The $p$-dimensional level set persistent homology functor $\HH_p\LL$ is colax $[0,\infty)$-equivariant.
\label{corollary:LPH}
\end{lemma}

\begin{proof}
Post-composing with the $p$-dimensional homology functor $\HH_{p}$ is a functor between GPM-categories with the same exponent poset $\Open(\M)$, and so, by Theorem~\ref{theorem:BubenikVsOurDefns} the functor $\HH_{p}\cdot-$ is colax $[0,\infty)$-equivariant.
The same holds for $\LL$.
Hence, their composition $\HH_p\LL$ is colax $[0,\infty)$-equivariant.
\end{proof}

\begin{corollary}[Curry~{\cite[Lemma~15.1.9]{Curry2014}}]
Level set persistence $\HH_p\LL$ is 1-Lipschitz with respect to the interleaving distance, i.e.
 \begin{equation*}
 	d_{\Vect_{k}^{(\R,\leq)}}(\HH_p\LL(\X),\HH_p\LL(\Y))\leq d_{\infty}(\X,\Y).
 \end{equation*}
\end{corollary}
\begin{proof}
The proof follows directly from Lemma~\ref{corollary:LPH} and Theorem~\ref{theorem:stability}.
\end{proof}

%-----------------------------
\section{Concluding Remarks}
\label{sec:MetaTheorem}
%-----------------------------

\subsection{Categories with a flow viewed as lax 2-functors}
\label{sec:Lax}
%---------------------------------------------
%________________________
In more generality one can formalize additional constructions defined on categories by enriching the structure of the category in the weak, strict or strong sense.
A common example of (weakly) strictly enriched categories are (bi)2-categories \cite{Leinster1998}.
Any (bi)2-category is a category (weakly) enriched over the monoidal category $\Cat$ of all small categories.
Then a lax 2-functor is a weakly $\Cat$-enriched functor and an oplax 2-natural transformation is a weakly $\Cat$-enriched natural transformation and so on.
Recall from \cite{MacLane1978} that a \textbf{metacategory} is any model of the first-order theory of categories, and a \textbf{category} is a metacategory whose  morphisms form sets.
Moreover, assuming the existence of one Grothendieck universe $U$, sets and categories in $U$ are called \textbf{small} and categories not in $U$  are called \textbf{large}.
Therefore, we have the large category $\Cat$ of all small categories, and the metacategory $\CAT$ of all (possibly large) categories.
The category $\Cat$ forms a 2-category with the composition operation $\circ$ on functors and natural transformations.
Similarly, $\CAT$ forms a meta 2-category with respect to the composition operation $\circ$.
We remark that the monoidal category $[0,\infty)$ can be seen also as a strict 2-category with one object denoted by $\textbf{B}[0,\infty)$.
Given these facts, we can define equivalently any category with a flow as a lax 2-functor
\begin{equation*}
	\textbf{B}[0,\infty)\xrightarrow{\hspace{1em}\hspace{1em}}\mathbf{CAT}
\end{equation*}
from the strict $2$-category $\textbf{B}[0,\infty)$ to the meta $2$-category $\mathbf{CAT}$ of all categories.

Let $(\CC,\TT)$ and $(\DD,\SS)$ be two categories with a flow.
Thinking of a category with a flow conceptually as a lax 2-functor from the $2$-category $\textbf{B}[0,\infty)$ to the meta $2$-category $\mathbf{CAT}$, we can equivalently define a colax $[0,\infty)$-equivariant functor  $\HH:(\CC,\TT)\to(\DD,\SS)$ as an oplax natural transformation between these lax 2-functors.
\begin{equation*}
	\begin{tikzcd}
		\textbf{B}[0,\infty)
		\arrow[bend left=65]{rr}[name=LUU]{(\CC,\TT)}
		\arrow[bend right=65]{rr}[name=LDD, below]{(\DD,\SS)}
		\arrow[shorten <=10pt,shorten >=10pt, Rightarrow,to path=(LUU) -- (LDD)\tikztonodes]{r}{\HH}
		\&
		\&
		\CAT
	\end{tikzcd}
\end{equation*}

\subsection{Summary theorem}
In this paper we showed that the interleaving construction assigns
\begin{itemize}
	\item  to each category with a flow $(\CC,\TT)$ the structure of a Lawvere metric space $(\CC,d_{(\CC,\TT)})$,
	\item  to each colax $[0,\infty)$-equivariant functor $\HH:\CC\to\DD$  the structure of a 1-Lipschitz map $\HH:(\CC,d_{(\CC,\TT)})\to(\DD,d_{(\DD,\SS)})$ with respect to the interleaving distances.
\end{itemize}
Define
\begin{itemize}
	\item $[0,\infty)\text{-}\ACT$
	to be the metacategory whose objects are categories with a flow and whose morphisms are colax $[0,\infty)$-equivariant functors, and
	\item  $[0,\infty]\text{-}\CAT$ to be the metacategory whose objects are $[0,\infty]$-enriched categories (Lawvere metric spaces) and whose morphisms are $[0,\infty]$-enriched functors (namely 1-Lipschitz maps).
\end{itemize}
We have the following meta-theorem for the interleaving construction.
\begin{theorem}
The interleaving construction
\begin{equation*}
\begin{array}{ccc}
[0,\infty)\text{-}\ACT&\xrightarrow{\II}&[0,\infty]\text{-}\CAT\\
(\CC,\TT) & \mapsto & (\CC,d_{(\CC,\TT)})\\
(\CC,\TT)\xrightarrow{\HH}(\DD,\SS)&\mapsto & (\CC,d_{(\CC,\TT)})\xrightarrow{\HH}(\DD,d_{(\DD,\SS)})
\end{array}
\end{equation*}
is functorial.
\end{theorem}
\begin{proof}
Indeed, if $\HH:(\CC,\TT)\to(\DD,\SS)$ and $\KK:(\DD,\SS)\to(\EE,\UU)$ are colax $[0,\infty)$-equivariant functors, then we get
	\begin{itemize}
		\item $\II[\KK]\II[\HH]=\KK\HH=\II[\KK\HH]$ and
		\item $\II[I_{[0,\infty)\text{-}\ACT}]=I_{[0,\infty]\text{-}\CAT}$, because the interleaving construction sends each identity equivalence $I_{(\CC,\TT)}$ to the identity isometry $I_{(\CC,d_{(\CC,\TT)})}$.
	\end{itemize}
\end{proof}

%--------------------------
\subsection{Discussion}
\label{sec:Discussion}
%--------------------------
In this paper, we gave a generalization of the interleaving distance, originally defined in the context of Topological Data Analysis (TDA), to the context of arbitrary categories with a flow.
We showed that many common metrics, not just those arising in TDA, can be viewed in this light.
We also investigated colax $[0,\infty)$-equivariant functors of categories with a flow, and provided a general stability result which specializes, in particular, to the seminal stability theorem for persistence diagrams.

In \cite{deSilva2016}, it was shown that another commonly used tool in TDA, the Reeb graph \cite{Reeb1946}, can be represented by particularly well behaved cosheaves.
In a subsequent paper \cite{deSilva2018}, we will show that the coherent $[0,\infty)$ framework created here also generalizes the stability theorem for the interleaving distance for Reeb graphs.
Further, we can define different oplax natural transformations to find bounds for the Reeb graph interleaving distance by the interleaving distance for simpler objects as exact computation of the Reeb graph interleaving distance is graph isomorphism hard.

Because this metric is so general, we expect that there are other example categories we have not yet thought of where this interleaving distance idea will be useful.
We should note that the infrastructure built here also has some immediate generalizations which we have not expanded on, such as replacing the functors or natural transformations with any combination of lax or oplax structures.
The restriction to lax functors and oplax natural transformations in this paper was merely due to the applications we are interested in.

\subsection*{Acknowledgements}
The authors gratefully acknowledge many helpful discussions and suggestions from Tom Leinster, Peter Bubenik, Justin Curry,
Amit Patel and Marco Varisco as well as the comments received from an anonymous reviewer which improved the state of the paper.
EM was partially supported by National Science Foundation through grants CMMI-1800466 and DMS-1800446.
%--------------------------
\bibliographystyle{plain}
\bibliography{Lax}

\begin{thebibliography}{10}

\bibitem{Bauer2014}
Ulrich Bauer and Michael Lesnick.
\newblock Induced matchings of barcodes and the algebraic stability of
  persistence.
\newblock In {\em Proceedings of the thirtieth annual symposium on
  Computational geometry}, page 355. ACM, 2014.

\bibitem{benabou1967introduction}
Jean B{\'e}nabou.
\newblock Introduction to bicategories.
\newblock In {\em Reports of the midwest category seminar}, pages 1--77.
  Springer, 1967.

\bibitem{Bubenik2015}
Peter Bubenik, Vin de~Silva, and Jonathan Scott.
\newblock Metrics for generalized persistence modules.
\newblock {\em Foundations of Computational Mathematics}, 15(6):1501--1531,
  2015.

\bibitem{Bubenik2014}
Peter Bubenik and Jonathan~A. Scott.
\newblock Categorification of persistent homology.
\newblock {\em Discrete \& Computational Geometry}, 51(3):600--627, 2014.

\bibitem{Chazal2009b}
Fr{\'e}d{\'e}ric Chazal, David Cohen-Steiner, Marc Glisse, Leonidas~J. Guibas,
  and Steve~Y. Oudot.
\newblock Proximity of persistence modules and their diagrams.
\newblock In {\em Proceedings of the 25th Annual Symposium on Computational
  Geometry}, SCG '09, pages 237--246, New York, NY, USA, 2009. ACM.

\bibitem{Chazal2016}
Fr\'{e}d\'{e}ric Chazal, Vin de~Silva, Marc Glisse, and Steve Oudot.
\newblock {\em The Structure and Stability of Persistence Modules}.
\newblock SpringerBriefs in Mathematics. Springer, 2016.

\bibitem{Cohen-Steiner2007}
David Cohen-Steiner, Herbert Edelsbrunner, and John Harer.
\newblock Stability of persistence diagrams.
\newblock {\em Discrete {\&} Computational Geometry}, 37(1):103--120, 2007.

\bibitem{Curry2014}
Justin Curry.
\newblock {\em Sheaves, Cosheaves and Applications}.
\newblock PhD thesis, University of Pennsylvania, 2014.

\bibitem{deSilva2016}
Vin de~Silva, Elizabeth Munch, and Amit Patel.
\newblock Categorified {R}eeb graphs.
\newblock {\em Discrete \& Computational Geometry}, pages 1--53, 2016.

\bibitem{deSilva2017}
Vin de~Silva, Elizabeth Munch, and Anastasios Stefanou.
\newblock Theory of interleavings on $[0,\infty)$-actegories.
\newblock {\em arXiv preprint arXiv:1706.04095}, 2017.

\bibitem{deSilva2018}
Vin {de Silva}, Elizabeth Munch, and Anastasios Stefanou.
\newblock A hom-tree lower bound for the {R}eeb graph interleaving distance.
\newblock Preprint, 2018.

\bibitem{Edelsbrunner2010}
Herbert Edelsbrunner and John Harer.
\newblock {\em Computational Topology: An Introduction}.
\newblock American Mathematical Society, 2010.

\bibitem{Edelsbrunner2002}
Herbert Edelsbrunner, David Letscher, and Afra Zomorodian.
\newblock Topological persistence and simplification.
\newblock {\em Discrete \& Computational Geometry}, 28(4):511--533, 2002.

\bibitem{janelidze2001note}
George Janelidze and G.~Maxwell Kelly.
\newblock A note on actions of a monoidal category.
\newblock {\em Theory Appl. Categ}, 9(61-91):02, 2001.

\bibitem{Kelly1974}
G.~Maxwell Kelly.
\newblock {\em Doctrinal adjunction}, pages 257--280.
\newblock Springer Berlin Heidelberg, Berlin, Heidelberg, 1974.

\bibitem{Leinster1998}
Tom Leinster.
\newblock Basic bicategories.
\newblock {\em arXiv preprint math.CT/9810017}, 589, 1998.

\bibitem{Lesnick2015}
Michael Lesnick.
\newblock The theory of the interleaving distance on multidimensional
  persistence modules.
\newblock {\em Foundations of Computational Mathematics}, 15(3):613--650, 2015.

\bibitem{MacLane1978}
Saunders Mac~Lane.
\newblock {\em Categories for the Working Mathematician}.
\newblock Springer-Verlag New York, 2nd edition, 1978.

\bibitem{morozov2013interleaving}
Dmitriy Morozov, Kenes Beketayev, and Gunther Weber.
\newblock Interleaving distance between merge trees.
\newblock In {\em Proceedings of TopoInVis}, 2013.

\bibitem{pareigis1977non}
Bodo Pareigis.
\newblock Non-additive ring and module theory. {II}. {C}-categories,
  {C}-functors and {C}-morphisms.
\newblock {\em Publ. Math. Debrecen}, 24(3-4):351--361, 1977.

\bibitem{Reeb1946}
Georges Reeb.
\newblock Sur les points singuliers d'une forme de {P}faff compl\`{e}ment
  int\'{e}grable ou d'une fonction num\'{e}rique.
\newblock {\em Comptes Rendus de L'Acad\'{e}mie ses S\'{e}ances}, 222:847--849,
  1946.

\bibitem{Singh2007}
Gurjeet Singh, Facundo M\'emoli, and Gunnar Carlsson.
\newblock Topological methods for the analysis of high dimensional data sets
  and {3D} object recognition.
\newblock In {\em Eurographics Symposium on Point-Based Graphics}, 2007.

\bibitem{AnastasiosThesis}
Anastasios Stefanou.
\newblock {\em Dynamics on Categories and Applications}.
\newblock PhD thesis, University at Albany, SUNY, 2018.

\bibitem{Zomorodian2004}
Afra Zomorodian and Gunnar Carlsson.
\newblock Computing persistent homology.
\newblock {\em Discrete \& Computational Geometry}, 33(2):249--274, November
  2004.

\end{thebibliography}

% \printbibliography

\end{document}